\documentclass[10pt,twoside, a4paper, english, reqno]{amsart}
\usepackage[dvips]{epsfig}
\usepackage{amscd}
\usepackage{caption}
\usepackage{a4wide}
\usepackage{amssymb}
\usepackage{amsthm}
\usepackage{amsmath}
\usepackage{bm} 
\usepackage{latexsym}
\usepackage{enumitem}
\usepackage{amsfonts}
\usepackage{dsfont}
\usepackage{tabularx}
\numberwithin{figure}{section}
\numberwithin{table}{section}
\usepackage{graphicx,caption}
\DeclareMathOperator{\sech}{sech}
\DeclareMathOperator{\csch}{csch}
\usepackage{calc}
\usepackage{graphicx}
\usepackage{amsmath,amsthm,bm,mathrsfs}
\usepackage[utf8]{inputenc} 
\usepackage{url}            
\usepackage{hyperref}
\hypersetup{
    colorlinks=true,
    linkcolor=blue,
    filecolor=red,      
    urlcolor=blue,
    citecolor=red,
    }
\usepackage{amsfonts}       
\usepackage{nicefrac}       
\usepackage{microtype}      
\theoremstyle{plain}
\usepackage{doi}

\usepackage{upref}
\usepackage{hyperref}
\usepackage{ulem}
\usepackage{enumitem}
\newcommand{\norm}[1]{\left\Vert#1\right\Vert}
\newcommand{\abs}[1]{\left|#1\right|}

\newcommand{\R}{\mathbb R}
\newcommand{\Z}{\mathbb{Z}}
\newcommand{\N}{\mathbb N}

\newtheorem{theorem}{Theorem}[section]

\newtheorem{lemma}[theorem]{Lemma}
\newtheorem{definition}[theorem]{Definition}

\newtheorem{remark}[theorem]{Remark}

\numberwithin{equation}{section}     
\numberwithin{figure}{section}
\numberwithin{table}{section}

\allowdisplaybreaks

\newcounter{asnr}

{\ifnum\value{asnr}=0 \stepcounter{asnr} 
  \begin{enumerate}[label=\textbf{A}.\arabic{enumi}]
    \else
    \begin{enumerate}[label=\textbf{A}.\arabic{enumi},resume] \fi}
{\end{enumerate}}

\newcounter{defnr}

{\ifnum\value{defnr}=0 \stepcounter{defnr} 
  \begin{enumerate}[label=\textbf{D}.\arabic{enumi}]
    \else
    \begin{enumerate}[label=\textbf{D}.\arabic{enumi},resume] \fi}
{\end{enumerate}}

\numberwithin{equation}{section} \allowdisplaybreaks

\title[finite difference schemes for fractional KdV equation]
{Fully discrete finite difference schemes for the Fractional Korteweg-de Vries equation}

\date{}

\author[M. Dwivedi]{Mukul Dwivedi}
\address[Mukul Dwivedi]{\newline
Department of Mathematics, 
	Indian Institute of Technology Jammu,
	Jagti, NH-44 Bypass Road, Post Office Nagrota,
	Jammu - 181221, India}
\email[]{mukul.dwivedi@iitjammu.ac.in}

\author[T. Sarkar]{Tanmay Sarkar}
\address[Tanmay Sarkar]{\newline
	Department of Mathematics, 
	Indian Institute of Technology Jammu,
	Jagti, NH-44 Bypass Road, Post Office Nagrota,
	Jammu - 181221, India}
\email[]{tanmay.sarkar@iitjammu.ac.in}



\thanks{}

\begin{document}
\begin{abstract}
In this paper, we present and analyze fully discrete finite difference schemes designed for solving the initial value problem associated with the fractional Korteweg-de Vries (KdV) equation involving the fractional Laplacian. We design the scheme by introducing the discrete fractional Laplacian operator which is consistent with the continuous operator, and posses certain properties which are instrumental for the convergence analysis.
Assuming the initial data \(u_0 \in H^{1+\alpha}(\mathbb{R})\), where \(\alpha \in [1,2)\), our study establishes the convergence of the approximate solutions obtained by the fully discrete finite difference schemes to a classical solution of the fractional KdV equation. Theoretical results are validated through several numerical illustrations for various values of fractional exponent $\alpha$. Furthermore, we demonstrate that the Crank-Nicolson finite difference scheme preserves the inherent conserved quantities along with the improved convergence rates.
\end{abstract}
\maketitle

\section{Introduction}
We consider a Cauchy problem associated with a fractional Korteweg–de Vries (KdV) equation which is a nonlinear, non-local dispersive equation that has gained prominence in the area of weakly nonlinear internal long waves. More precisely, the equation reads
\begin{equation}\label{fkdv}
\begin{cases}
      u_t+\left(\frac{u^2}{2}\right)_x-(-\Delta)^{\alpha/2}u_x=0,  \qquad & (x,t) \in Q_T:= \mathbb{R}\times (0,T],\\
     u(x,0) = u_0(x), \qquad & x \in \mathbb{R},
\end{cases}
\end{equation}
where $T>0$ is fixed, $u_0$ is the prescribed initial condition and $u:Q_T\rightarrow \R$ represents the unknown solution.
The non-local operator $-(-\Delta)^{\alpha/2}$ in \eqref{fkdv} corresponds to the fractional Laplacian, a distinctive feature that introduces non-locality into the dynamics of the equation. The parameter $\alpha \in [1, 2)$ dictates the fractional order and plays a pivotal role in shaping the behavior of the solutions. For $\phi\in C_{c}^{\infty}(\mathbb{R})$, the fractional Laplacian is defined as:   
 \begin{equation}\label{fracL}
     -(-\Delta)^{\alpha/2}[\phi](x) = c_\alpha \text{P.V.}\int_{\mathbb{R}} \frac{\phi(y)-\phi(x)}{|y-x|^{1+\alpha}}\,dy,
 \end{equation}
 for some constant $c_\alpha>0$, which is described in \cite{di2012hitchhikers,dutta2021operator,koley2021multilevel,kwasnicki2017ten}. 
 
 Fractional Laplacian involved differential equation has emerged as a powerful and versatile tool across diverse scientific, economic and engineering domains, owing to the computational advantages inherent in integral operator and fractional-order derivative. The widespread adoption of fractional Laplacian is attributed to their effectiveness in handling localized computations, ranging from image segmentation to the study of water flow in narrow channels and the complexities of plasma physics. The historical evolution of fractional-order operators further enhances their efficacy, providing a solid foundation for modeling various scenarios with precision and versatility. As these operators continue to play a pivotal role in advancing computational methodologies and their impact on scientific and applied disciplines remains profound.

Researchers have extensively studied the well-posedness, both locally and globally, of the fractional KdV equation \eqref{fkdv}. 
We will not be able to address all the literature here, but to mention the relevant ones which are related to the current work.
When $\alpha = 2$, the equation \eqref{fkdv} simplifies to the classical KdV equation \cite{Bona1975KdV, dutta2021operator, korteweg1895xli}, a well-established model for solitons and nonlinear wave phenomena. Recent decades have seen in-depth investigations into the well-posedness of the Cauchy problem associated to the KdV equation, for instance, see \cite{kato1983cauchy,killip2019kdv} and references therein. For $\alpha = 1$, the equation \eqref{fkdv} becomes the Benjamin-Ono equation \cite{fokas1981hierarchy,fonseca2013ivp,kenig1994generalized}, designed to describe weakly nonlinear internal long waves. Well-posedness theory for the Benjamin-Ono equation is well-established in \cite{ponce1991global,kenig1994generalized, tao2004global} and references therein. 
The well-posedness of generalized dispersion model \eqref{fkdv} has also been studied in the literature. For instance, Kenig et al. \cite{kenig1991well,kenig1993cauchy} showed the existence and uniqueness of \eqref{fkdv} for $u_0\in H^s(\R),~ s>3/4$. Afterwards, 
Herr et al. \cite{herr2010differential} proved the global well-posedness by using frequency dependent renormalization technique for $L^2$ data. For more related work, one may refer to \cite{ehrnstrom2019enhanced, kenig2020unique,molinet2018well}.

Several numerical methods have been proposed for solving the  equation \eqref{fkdv}. In particular, for the case $\alpha=2$,
fully discrete finite difference schemes have been developed by  \cite{Amorim2013mKdV,dwivedi2023convergence,holden2015convergence,wang2021high}. Courtes et al. \cite{courtes2020error} established error estimates for a finite difference method while Skogestad et al. \cite{skogestad2009boundary} conducted a comparison between finite difference schemes and Chebyshev methods. Li et al. \cite{li2006high} proposed high-order compact schemes for linear and nonlinear dispersive equations. Apart from this, Galerkin schemes have been designed in \cite{dutta2015convergence, dutta2016note}. In case of $\alpha=1$,
Thomee et al. \cite{thomee1998numerical} introduced a fully implicit finite difference scheme, and Dutta et al. \cite{dutta2016convergence} established the convergence of the fully discrete Crank-Nicolson scheme. Furthermore, Galtung \cite{galtung2018convergent} designed a Galerkin scheme and proved its convergence to the weak solution.
However, for the range $\alpha\in(1,2)$, there are limited literature on the numerical methods of \eqref{fkdv}. An operator splitting scheme is introduced in \cite{dutta2021operator} and recently, a Galerkin scheme is developed in \cite{dwivedi2023stability}. Hereby our focus is not to develop the most efficient scheme, rather to design the convergent finite difference schemes. Since the convergence analysis is based on the semi-discrete work of Sj\"{o}berg \cite{sjoberg1970korteweg}, there is a need of fully discrete schemes which are computationally efficient.

Motivated by the work of Holden et al. \cite{holden2015convergence} and Dutta et al. \cite{dutta2015convergence} for the KdV equation and Benjamin-Ono equation respectively, we have developed two finite difference schemes to obtain the approximate solutions of the fractional KdV equation \eqref{fkdv}. The proposed schemes differ in their temporal discretization along with the averages of approximate solution considered in the convective term. The main objective of this paper is to show that the approximate solutions obtained by the difference schemes converge uniformly to the classical solution of the fractional KdV equation \eqref{fkdv} in $C(\R\times[0,T])$ provided the initial data is in $H^{1+\alpha}(\R)$.

 In this paper, our focus is in the discretization of the fractional Laplacian and its subsequent numerical implementation. The derivation of two schemes involves a treatment of the nonlinear terms, with theoretical foundations and proof ideas borrowed from the seminal works of Holden et al. \cite{holden2015convergence} and Dutta et al. \cite{dutta2015convergence}. However, the main challenge revolves around the precise computation of the fractional Laplacian and establishing its discrete properties. 
Moreover, the Crank-Nicolson finite difference scheme is designed in such a way that it exhibits conservation properties analogous to those observed in classical solutions of the fractional KdV equation, as stipulated in \cite{kenig1994generalized}. These conserved quantities include mass, momentum and energy, defined as:
\begin{align*}
        C_1(u): & = \int_{\R} u(x,t)~dx, \quad C_2(u):= \int_{\R} u^2(x,t)~dx,\\
         C_3(u):&= \int_{\R} \left((-(-\Delta)^{\alpha/4}u)^2 - \frac{u^3}{3}\right)(x,t)~dx, \qquad \alpha \in [1,2).
\end{align*}
Through numerical illustrations, we empirically demonstrate the superior performance of the Crank-Nicolson scheme over the Euler implicit scheme, aligning with expectations.

The organization of the paper is as follows: In Section \ref{sec2}, we present the necessary notations involving numerical discretization and introduce discrete operators. In addition, we show that discrete fractional Laplacian is consistent and satisfies certain properties. In Section \ref{sec3}, we introduce an implicit finite difference scheme which is solvable, stable and converges to the weak solution of \eqref{fkdv}. In Section \ref{sec4}, we design a Crank-Nicolson finite difference scheme. The scheme is shown to be stable and convergent for the initial data in $H^{1+\alpha}(\R)$.
We validate our theoretical results through the numerical illustrations in Section \ref{sec5} for various values of $\alpha\in[1,2]$. Finally, we end up with some concluding remarks in Section \ref{sec6}.
    
\section{Numerical discretization and discrete operators}\label{sec2}
\label{sec:headings}
In establishing a foundational framework, we introduce some essential notations, definitions, and inequalities. Let $\Delta x$ and $\Delta t$ denote the small mesh sizes corresponding to the spatial and temporal variables respectively. We discretize both the spatial and temporal axes using these mesh sizes, defining $x_j = j\Delta x$ for $j\in\mathbb{Z}$ and $t_n = n\Delta t$ for $n\in\N_0:=\N\cup\{0\}$. Subsequently, a fully discrete grid function in both space and time is defined by $u : \Delta x \mathbb{Z} \times \Delta t \mathbb{N}_0 \to \mathbb{R}^{\mathbb{Z}}$ with $u(x_j,t_n) =: u_j^n$, and we represent $(u^n)_{j\in\mathbb{Z}}$ as $u^n$ for brevity. Notably, $u^n$ is constructed as a spatial grid function.

We further introduce difference operators for a function $v:\mathbb{R}\to\mathbb{R}$. These operators are defined as follows:
\begin{equation}\label{frdiffs}
     D_{\pm}v(x) = \pm\frac{1}{\Delta x}\big(v(x\pm\Delta x) - v(x)\big), \qquad D = \frac{1}{2}(D_+ +D_-).
\end{equation}
We also introduce the shift operators $$S^{\pm}v(x) = v(x\pm\Delta x),$$ and the averages $\widetilde{v}(x)$ and $\bar{v}(x)$ are defined by
\[\widetilde{v}(x) := \frac{1}{3}\left(S^+v(x) + v(x) + S^-v(x) \right), \qquad \Bar{v}(x) := \frac{1}{2}(S^+ + S^-)v(x).\]
The difference operators satisfy the following product formulas
    \begin{align}
        \label{prod1}D(vw) &= \bar{v}Dw + \bar{w}Dv, \\
        \label{prod2}D_{\pm}(vw) & = S^{\pm}vD_{\pm}w + wD_{\pm}v = S^{\pm}w D_{\pm}v+vD_{\pm}w.
    \end{align}
We consider the usual $\ell^2$-inner product, denoted as $\langle \cdot,\cdot \rangle$, defined by  
\begin{equation}\label{normdef}
    \langle v,w\rangle = \Delta x\sum_{j\in\mathbb{Z}}v_jw_j, \qquad \norm{v}=\norm{v}_2=\langle v,v \rangle^{1/2}, \qquad v, w \in \ell^2.
\end{equation}
As a consequence, we have the following estimates 
    \begin{equation}\label{supnorm}
        \|v\|_\infty :=\max\{|v_j| : j\in\Z\}\leq\frac{1}{\Delta x^{1/2}}\|v\|, \qquad  \|Dv\| \leq\frac{1}{\Delta x}\|v\|, \quad v \in \ell^2.
    \end{equation}
    
    The above inequalities are straightforward to establish, given that $v \in \ell^2$, implying $|v_j|$ tends to $0$ as $j \xrightarrow[]{} \infty$. Thus, taking the maximum yields a simple proof. Several properties of the difference operators in relation to their $\ell^2$-inner product can be derived. Let $v, w \in \ell^2$. Then
\begin{equation}\label{idenDD}
    \langle v, D_{\pm}w \rangle = -\langle D_{\mp}v,w \rangle, \qquad \langle v, Dw\rangle = - \langle Dv, w \rangle.
\end{equation}
 Furthermore, employing the product formulas \eqref{prod1} and \eqref{prod2} along with the properties \eqref{idenDD}, we get the following identities:
\begin{align}
    \label{Dz_1z_2} \left\langle D(vw), w \right\rangle &= \frac{\Delta x}{2} \left\langle D_+vD(w), w \right\rangle + \frac{1}{2} \left\langle S^-wDv, w \right\rangle, \\
    \label{D3z_1z_2} D_+D_- (vw) &= D_-vD_+w + S^-vD_+D_-w + D_+vD_+w + wD_+D_-v.
\end{align}
\subsection{Discretization of fractional Laplacian}
In the pursuit of a discretization of the fractional Laplacian, we consider the definition of fractional Laplacian. Let $u\in\mathcal{S(\R)}$, where $\mathcal{S}(\R)$ denotes the Schwartz space. Then
\begin{align}
      \nonumber-(-\Delta)^{\alpha/2}[u](x) &= c_\alpha \text{P.V.}\int_{\mathbb{R}} \frac{u(y)-u(x)}{|y-x|^{1+\alpha}}\,dy\\
      \label{FracLapequ}&=\frac{1}{2}c_\alpha \int_{\mathbb{R}} \frac{u(x+y)-2u(x)+u(x-y)}{|y|^{1+\alpha}}\,dy.
\end{align}
The above equivalence follows by the standard change of variable formula. In fact, \eqref{FracLapequ} is quite useful to remove the singularity of the integral. For any smooth function $u$, a second order Taylor series expansion yields
  \begin{equation}\label{errorfor}
      \frac{u(x+y)-2u(x)+u(x-y)}{|y|^{1+\alpha}}\leq \frac{\|\partial_x^2u\|_{L^\infty (\R)}}{|y|^{\alpha-1}},
  \end{equation}
which is integrable near origin.
  
Our approach to discretize the fractional Laplacian begins with the consideration of even indices. More precisely, 
  \begin{equation*}
      \begin{split}
          (-(-\Delta)^{\alpha/2}u)_j &=  c_\alpha \text{P.V.}\int_{\mathbb{R}} \frac{u(y)-u_j}{|y-x_j|^{1+\alpha}}\,dy\\
          &= c_\alpha \sum_{k=\text{even}} \int_{x_k}^{x_{k+2}} \frac{u(y)-u_j}{|y-x_j|^{1+\alpha}}\,dy,
      \end{split}
  \end{equation*}
where $u(x_j):=u_j$. Employing the midpoint formula on the integrals, we arrive at the discrete approximation:
 \begin{equation*}
     (-(-\Delta)^{\alpha/2}u)_j \approx  c_\alpha \sum_{k=\text{odd}} 2\Delta x \frac{u_k-u_j}{|x_k-x_j|^{1+\alpha}}
 \end{equation*}
which can be rewritten as
  \begin{equation*}
     (-(-\Delta)^{\alpha/2}u)_j \approx \frac{2c_\alpha}{\Delta x^\alpha}  \sum_{k=\text{odd}}  \frac{u_k-u_j}{|k-j|^{1+\alpha}}.
 \end{equation*}
Similarly, by considering the odd indices, we end up with
\begin{equation*}
     (-(-\Delta)^{\alpha/2}u)_j \approx \frac{2c_\alpha}{\Delta x^\alpha}  \sum_{k=\text{even}}  \frac{u_k-u_j}{|k-j|^{1+\alpha}}.
\end{equation*}
Consequently, the combining of the above results leads to the definition of the discrete fractional Laplacian, denoted as $\mathbb{D}^{\alpha}$, expressed as follows:
   \begin{align}\label{DisLap}
   \mathbb{D}^{\alpha}(u)_j &= \frac{c_\alpha}{\Delta x^\alpha}  \sum_{k\neq j}  \frac{u_k-u_j}{|k-j|^{1+\alpha}} (1-(-1)^{j-k}).
  \end{align}
Another approach to define the discrete fractional Laplacian can be incorporating the definition given by Nezza et al. \cite[Lemma 3.2]{di2012hitchhikers} in which the singular integral is represented by a weighted second order differential quotient.
\begin{lemma}\label{equi_frac}
    Let $\alpha\in[1,2)$ and $\mathbb{D}^{\alpha}$ be the discrete fractional Laplacian defined by \eqref{DisLap}. Then, for any $u\in\mathcal{S}(\R)$,
    \begin{equation}\label{Nonsigudef}
        \mathbb{D}^{\alpha}(u)_j  = \frac{c_\alpha}{2}  \sum_{i}\int_{x_{2i}}^{x_{2i+2}}  \frac{u(x_j+x_{2i+1})+u(x_j-x_{2i+1})-2u(x_{j})}{|x_{2i+1}|^{1+\alpha}}\,dx.
    \end{equation}
\end{lemma}
\begin{proof}
We consider the equation \eqref{DisLap} and by the change of variable $k-j=i$, it results into
\begin{equation*}
    \begin{split}
        2\sum_{i\neq 0} \frac{u_{j+i}-u_j}{|i|^{1+\alpha}}(1-(-1)^{i}) &= \sum_{i\neq 0} \frac{u_{j+i}-u_j}{|i|^{1+\alpha}}(1-(-1)^{i})+\sum_{i\neq 0} \frac{u_{j+i}-u_j}{|i|^{1+\alpha}}(1-(-1)^{i})\\
        &= \sum_{i\neq 0} \frac{u_{j+i}-u_j}{|i|^{1+\alpha}}(1-(-1)^{i}) + \sum_{i\neq 0} \frac{u_{j-i}-u_j}{|i|^{1+\alpha}}(1-(-1)^{i})\\
        &= \sum_{i} \frac{u_{j+i}+u_{j-i}-2u_j}{|i|^{1+\alpha}}(1-(-1)^{i}),
    \end{split}
\end{equation*}
as only odd $i$ survives, thus
\begin{align*}
    \nonumber 2\sum_{i\neq 0} \frac{u_{j+i}-u_j}{|i|^{1+\alpha}}(1-(-1)^{i}) &= 2\Delta x^{1+\alpha}\sum_{i}\frac{u(x_j+x_{2i+1})+u(x_j-x_{2i+1})-2u(x_{j})}{|x_{2i+1}|^{1+\alpha}}\\
    &= \Delta x^{\alpha}\sum_{i}\int_{x_{2i}}^{x_{2i+2}}\frac{u(x_j+x_{2i+1})+u(x_j-x_{2i+1})-2u(x_{j})}{|x_{2i+1}|^{1+\alpha}}\,dx.
\end{align*}
Hence the result follows.
\end{proof}
The Lemma \ref{equi_frac} will help us to show the relation between continuous and discrete fractional Sobolev norms. Let us define the $h^2$-norm for the given grid function $u^n$ by 
    \begin{equation*}
        \|u^n\|_{h^2} = \|u^n\| + \|D_+u^{n}\|+\|D_+D_-u^{n}\|.
    \end{equation*} 

There exists an alternative definition of the fractional Laplacian for a function $u\in\mathcal{S}(\R)$, provided by Yang et al. \cite[Lemma 1]{yang2010numerical}, expressed by the following equation:
\begin{equation}\label{fracLfrac}
    -(-\Delta)^{\alpha/2}u(x) = \frac{\partial^\alpha}{\partial|x|^\alpha}u(x) =  -\frac{{}_{-\infty}D_{x}^\alpha u(x) + {}_{x}D_{\infty}^\alpha u(x)}{2\cos\left(\frac{\alpha\pi}{2}\right)}, \quad 1<\alpha<2,
\end{equation}
where ${}_{-\infty}D_{x}^\alpha$ and ${}_{x}D_{\infty}^\alpha$ are defined as the left- and right-side Riemann–Liouville derivatives \cite{yang2010numerical}. Following Xu et al. \cite{xu2014discontinuous} and Ervin et al. \cite{ervin2006variational}, we introduce the following definitions of norms, aiding in the formulation of a norm on the fractional Laplacian \eqref{fracLfrac}.

\begin{definition}
    We define the semi-norms 
    \begin{align}\label{SeminormfracL}
        \abs{u}_{J_L^\alpha{(\mathbb R)}} = \norm{{}_{-\infty}D_x^\alpha u}_{L^2(\mathbb{R})}, 
        \qquad \abs{u}_{J_R^\alpha(\mathbb{R})} = \norm{{}_xD_{\infty}^\alpha u}_{L^2(\mathbb{R})},
    \end{align}
    and the norms
    \begin{align}\label{NormfracL}
       \norm{u}_{J_L^\alpha{(\mathbb R)}}  = \abs{u}_{J_L^\alpha{(\mathbb{R})}} + \norm{u}_{L^2(\mathbb{R})}, 
       \qquad \norm{u}_{J_R^\alpha(\mathbb{R})}  = \abs{u}_{J_R^\alpha(\mathbb{R})} + \norm{u}_{L^2(\mathbb{R})},
    \end{align}
    and let $J_L^\alpha{(\mathbb{R})}$ and $J_R^\alpha(\mathbb{R})$ denote the closure of $C_c^{\infty}(\mathbb{R})$ with respect to $\norm{\cdot}_{J_L^\alpha{(\mathbb{R})}}$ and $\norm{\cdot}_{J_R^\alpha(\mathbb{R})}$ respectively.
\end{definition}

It is evident from the definition \eqref{fracLfrac} that
\begin{equation}\label{fracnorm}
    \norm{-(-\Delta)^{\alpha/2}u(x)}_{L^2(\mathbb{R})} = C\left(\norm{{}_{-\infty}D_x^\alpha u}_{L^2(\mathbb{R})}+ \norm{{}_xD_\infty^\alpha u}_{L^2(\mathbb{R})}\right),\quad 1<\alpha<2,
\end{equation}
where $C$ is a constant independent of $u$. We wish to establish a connection between fractional Sobolev spaces and the aforementioned fractional derivative spaces. According to the result by Ervin et al. \cite[Theorem 2.1]{ervin2006variational}, fractional derivative spaces $J_L^\alpha{\mathbb{R}}$ and $J_R^\alpha(\mathbb{R})$ with respect to norms defined by \eqref{NormfracL}, are equal to the fractional Sobolev space $H^\alpha(\mathbb{R})$. 
Therefore, using the norm on the fractional Laplacian \eqref{fracnorm}, we conclude that the function $u\in L^2(\R)$ with $-(-\Delta)^{\alpha/2}\partial_xu\in L^2(\mathbb{R})$ implies $ u\in H^{1+\alpha}(\mathbb{R})$.
    

Subsequently, we define the norm for discrete fractional derivative space $h^{1+\alpha}$ as 
\begin{equation*}
        \|u^n\|_{h^{1+\alpha}} = \|u^n\| + \|D_+u^{n}\|+\|D_+D_-u^{n}\| + \|\mathbb{D}^{\alpha}Du^n\|.
    \end{equation*}   

Now, we present a lemma establishing a fundamental connection between continuous and discrete derivative norms. Given the frequent use of this lemma, we include a proof for the sake of completeness.

\begin{lemma}\label{D_C}
    Let $u\in H^{1+\alpha}(\R)$, $\alpha\in [1,2)$. Then there is a constant $C$ such that the discrete evaluation $\{u(x_j)\}_j$ satisfies
    \begin{equation}
        \|u_{\Delta x} \|_{h^{1+\alpha}} \leq C \|u\|_{H^{1+\alpha}(\R)}.
    \end{equation}
\end{lemma}
    \begin{proof}
        We begin by observing that the discrete fractional Laplacian $\mathbb D^\alpha$ commutes with other difference operators and derivatives. By the definition of $\ell^2$-norm and $\mathbb D^\alpha$, we have 
        \begin{align*}
            \norm{\mathbb D^\alpha Du}^2 &\leq C \norm{\mathbb D^\alpha \partial_x u}^2\\
            &= C\Delta x \sum_j \left(\mathbb D^\alpha (\partial_x u)_j\right)^2 \\
            & = C\Delta x \sum_j \left(\frac{c_\alpha}{2}\sum_{i}\int_{x_{2i}}^{x_{2i+2}}\frac{\partial_x u(x_j+x_{2i+1})+\partial_x u(x_j-x_{2i+1})-2\partial_x u(x_{j})}{|x_{2i+1}|^{1+\alpha}}\,dx\right)^2 \\
            & \leq C\Delta x \sum_j \left(\frac{c_\alpha}{2}\int_{\R}\frac{\partial_x u(x_j+y)+\partial_x u(x_j-y)-2\partial_x u(x_{j})}{|y|^{1+\alpha}}\,dy\right)^2 \\
            &\leq C \norm{u}_{H^{1+\alpha}(\R)}.
        \end{align*}
       We have used midpoint formula in the above calculation. Hence the result follows.
    \end{proof}
Please note that from \cite[Lemma 2.1]{dutta2016convergence}, we have $\norm{u_{\Delta x}}_{h^2} \leq C \norm{u}_{H^2(\R)}$. We establish certain properties of the discrete fractional Laplacian \eqref{DisLap} in the following lemma, mirroring the discrete counterparts of properties of the fractional Laplacian introduced in \cite{dwivedi2023stability}.

\begin{lemma}\label{P_fracLap}
The discrete fractional Laplacian $\mathbb{D}^{\alpha}$, $\alpha\in[1,2)$ defined by \eqref{DisLap}, exhibits linearity and possesses the following noteworthy properties for any pair of grid functions $u,v\in \ell^2$:
\begin{enumerate}[label=(\roman*)]
    \item \label{itm2.1}(Symmetry) The discrete fractional Laplacian exhibits symmetry:
    \[  \langle\mathbb{D}^{\alpha} u, v\rangle = \langle u,\mathbb{D}^{\alpha} v\rangle.\]
    \item \label{itm2.2}(Translation invariant and Skew-symmetry) The discrete fractional Laplacian commutes with the difference operator:
    \[ \langle \mathbb{D}^{\alpha}Du,v\rangle = -\langle u,\mathbb{D}^{\alpha}Dv\rangle = \langle D\mathbb{D}^{\alpha}u,v\rangle.\]
    \item \label{itm2.3}The discrete fractional Laplacian further satisfies:
    \[  \langle \mathbb{D}^{\alpha}Du,u\rangle = 0.\]
\end{enumerate}
\end{lemma}

\begin{proof} Let $u,v\in\ell^2$. Then the definition of $\ell^2$-inner product and discrete fractional Laplacian \eqref{DisLap} provide
\begin{align*}
    \langle\mathbb{D}^{\alpha} u, v\rangle &= \Delta x \sum_j \frac{c_\alpha}{\Delta x^\alpha}\sum_{k\neq j}  \frac{v_ju_k-v_ju_j}{|k-j|^{1+\alpha}} (1-(-1)^{j-k})\\
    & = \Delta x \sum_j \frac{c_\alpha}{\Delta x^\alpha}\sum_{k\neq j}  \frac{v_k-v_j}{|k-j|^{1+\alpha}} (1-(-1)^{j-k})u_j\\
    & = \langle u,\mathbb{D}^{\alpha} v\rangle,    
\end{align*}
where we have used the change of variable. Thus \ref{itm2.1} follows. Similarly, we have
\begin{equation*}
     \langle \mathbb{D}^{\alpha}(Du),v\rangle = \frac{c_\alpha}{\Delta x^\alpha}  \sum_{j}\sum_{k\neq j}  \frac{Du_k-Du_j}{|k-j|^{1+\alpha}} (1-(-1)^{j-k})v_j.
\end{equation*}
This can be presented as 
\begin{equation*}
\begin{split}
       \langle\mathbb{D}^{\alpha}(Du),v\rangle &= \Delta x \sum_j \frac{c_\alpha}{\Delta x^\alpha}\sum_{k\neq j}  \frac{(Du_k-Du_j)v_j}{|k-j|^{1+\alpha}} (1-(-1)^{j-k})\\
       &=  \Delta x \sum_j \frac{c_\alpha}{\Delta x^\alpha}\sum_{k\neq j} \frac{v_j(u_{k+1}-u_{k-1})-v_j(u_{j+1}-u_{j-1})}{2\Delta x|k-j|^{1+\alpha}} (1-(-1)^{j-k}).\\
\end{split}
\end{equation*}
Suitable change of variables in $j$ and $k$ implies
\begin{equation*}
\begin{split}
    \langle\mathbb{D}^{\alpha}(Du),v\rangle &= \Delta x \sum_j \frac{c_\alpha}{\Delta x^\alpha}\sum_{k\neq j}  \frac{u_k(v_{j-1}-v_{j+1})-u_j(v_{j-1}-v_{j+1})}{2\Delta x|k-j|^{1+\alpha}} (1-(-1)^{j-k})\\
    &= -\Delta x \sum_j \frac{c_\alpha}{\Delta x^\alpha}\sum_{k\neq j} \frac{(Dv_k-Dv_j)u_j}{|k-j|^{1+\alpha}} (1-(-1)^{j-k})\\
    &= -\langle\mathbb{D}^{\alpha}(u),Dv\rangle\\
    &= \langle D\mathbb{D}^{\alpha}(u),v\rangle, 
\end{split}
\end{equation*}
where we have used the property of symmetric difference operator over $\ell^2$-inner product, hence \ref{itm2.2} follows. Property \ref{itm2.3} follows from \ref{itm2.1} and \ref{itm2.2} by choosing $v=u$ in $\ref{itm2.2}$.
\end{proof}
In the following lemma, we show that the discretization of fractional Laplacian is consistent. 
\begin{lemma}\label{consit}
Let $\phi\in C_c^4(\mathbb{R})$. Define a piece-wise constant function $d$ by 
\[  d(x) = d_j = \mathbb{D}^{\alpha}(\phi)(x_j) \text{ for } x\in[x_j,x_{j+1}), \quad j\in\mathbb{Z}.   \]
Then
\begin{equation*}
     \lim_{\Delta x\xrightarrow[]{}0}\|(-(-\Delta)^{\alpha/2})(\phi) - d\|_{L^2(\mathbb{R})} = 0.
\end{equation*}
\end{lemma}
\begin{proof}
Let $\Tilde{d}$ be an auxiliary function defined by
\[ \Tilde{d}(x) = \Tilde{d}_j = (-(-\Delta)^{\alpha/2})(\phi)(x_j) \text{ for } x\in[x_j,x_{j+1}), \qquad j\in\mathbb{Z}.\]
By the triangle inequality, we have
\[  \|(-(-\Delta)^{\alpha/2})(\phi) - d\|_{L^2(\mathbb{R})}  \leq \|(-(-\Delta)^{\alpha/2})(\phi) - \Tilde{d}\|_{L^2(\mathbb{R})} + \|\Tilde{d} - d\|_{L^2(\mathbb{R})}.\]
Let us estimate the first term on the right hand side by
\begin{equation*}
    \begin{split}
        \Big\|\left(-(-\Delta)^{\alpha/2}\right)(\phi) - \Tilde{d}\Big\|_{L^2(\mathbb{R})}^2 &= \sum_{j}\int_{x_j}^{x_{j+1}}\left(\left(-(-\Delta)^{\alpha/2}\right)(\phi)(x) - \left(-(-\Delta)^{\alpha/2}\right)(\phi)(x_j)\right)^2\,dx\\
        &= \sum_{j}\int_{x_j}^{x_{j+1}}\left(\int_{x_j}^x \left(\left(-(-\Delta)^{\alpha/2}\right)(\phi)\right)'(\xi)\,d\xi\right)^2\,dx\\
        &= \sum_{j}\int_{x_j}^{x_{j+1}}\left(\int_{x_j}^x 1\cdot\left(\left(-(-\Delta)^{\alpha/2}\right)(\phi')\right)(\xi)\,d\xi\right)^2\,dx\\
        &\leq \sum_{j}\int_{x_j}^{x_{j+1}}\left(\int_{x_j}^{x_{j+1}} \left(\left(-(-\Delta)^{\alpha/2}\right)(\phi')(\xi)\right)^2\,d\xi\right)\,(x-x_j)\,dx\\
        &= \frac{\Delta x^2}{2}\|(-(-\Delta)^{\alpha/2})(\phi')\|^2_{L^2(\mathbb{R})}.
    \end{split}
\end{equation*}
The second term can be estimated by
\begin{equation*}
    \begin{split}
        \|\Tilde{d} - d\|_{L^2(\mathbb{R})}^2 &= \Delta x \sum_j (d_j - \Tilde{d_j})^2\leq \Delta x \sum_{|j|\leq J} (d_j - \Tilde{d_j})^2 + 2\Delta x \sum_{|j|>J} (d_j^2 + \Tilde{d_j^2})\\
        &=: K_1 + K_2.
    \end{split}
\end{equation*}
From the equations \eqref{FracLapequ} and \eqref{Nonsigudef}, we have
\begin{equation*}
    d_j - \Tilde{d_j} = \sum_{i}\left(\int_{x_{2i}}^{x_{2i+2}}\xi(x_j,x_{2i+1})\,dy-\int_{x_{2i}}^{x_{2i+2}}\xi(x_j,y)\,dy\right),
\end{equation*}
where $\xi(x,y) =\frac{c_\alpha}{2} (\phi(x+y)-2\phi(x)+\phi(x-y))/|y|^{1+\alpha}$. 

Employing the midpoint quadrature error formula and bound \eqref{errorfor}, we establish the following estimate:
\begin{equation*}
  \begin{split}
      \left|\int_{x_{2i}}^{x_{2i+2}}\xi(x_j,x_{2i+1})\,dy-\int_{x_{2i}}^{x_{2i+2}}\xi(x_j,y)\,dy\right|&\leq c_\alpha C \Delta x^{4-\alpha} \frac{\|\phi^{(4)}\|_{L^{\infty}}}{|2i+1|^{\alpha-1}},
  \end{split}
\end{equation*}
where $C$ is a constant independent of $\Delta x$ and $j$. Furthermore, since $\phi$ has compact support, the summation over $i$ contains only finite number of terms, say $N_\phi/\Delta x$, independent of $j$. Consequently,
\begin{equation*}
    |d_j-\tilde{d_j}|\leq N_\phi c_\alpha C \Delta x^{3-\alpha} \|\phi^{(4)}\|_{L^{\infty}(\R)}M_\alpha,
\end{equation*}
where $M_\alpha$ is an upper bound for $\sum_{|i|\leq K} \frac{1}{|2i+1|^{\alpha-1}}$ with finite $K$ arising from the support of $\phi$. Thus we have
\begin{equation*}
    K_1 \leq N_\phi M_\alpha c_\alpha C \Delta x^{3-\alpha} \|\phi^{(4)}\|_{L^{\infty}(\R)}.
\end{equation*}
Given the finite-ness of $\sum_{j} d_j^2$ and $\sum_j \tilde{d_j^2}$, a judicious choice of a sufficiently large $J$ ensures a small $K_2$ and subsequently, $\Delta x\to 0$ leads to a small $K_1$. Consequently, $\|d_j-\tilde{d_j}\|_{L^2(\R)}\to0$ as $\Delta x \to 0$. Hence the consistency of discrete fractional Laplacian is established.
\end{proof}

\section{Fully discrete semi-implicit scheme}\label{sec3}
We propose the following Euler implicit temporal discretized finite difference scheme to obtain approximate solutions of the fractional KdV equation \eqref{fkdv}:
\begin{equation}\label{FDscheme}
    u_j^{n+1} = \bar u_j^n -\Delta t \bar u_j^n Du_j^n - \Delta t \mathbb{D}^{\alpha}(Du^{n+1}_j),\qquad n\in\mathbb{N}_0,\hspace{0.1cm}j\in\mathbb{Z}.
\end{equation}
For the initial data, we have 
\begin{equation*}
    u_j^0 = u_0(x_j), \qquad j\in \mathbb{Z}.
\end{equation*}
\begin{remark}
    We must ensure that the above scheme \eqref{FDscheme} is solvable with respect to $u^{n+1}$. Solvability can be achieved by rewriting the scheme \eqref{FDscheme} as follows:
    \begin{equation*}
   (1+ \Delta t \mathbb{D}^{\alpha}D)u_j^{n+1} = \bar u_j^n -\Delta t \bar u_j^n Du_j^n.
\end{equation*}
Taking the inner product with $u^{n+1}$, we get:
\begin{align*}
    \langle (1+ \Delta t \mathbb{D}^{\alpha}D)u^{n+1},u^{n+1}\rangle & = \norm{u^{n+1}}^2 + \Delta t\langle \mathbb{D}^{\alpha}Du^{n+1},u^{n+1}\rangle = \norm{u^{n+1}}^2.
\end{align*}
Hence we obtain
\begin{equation*}
    \norm{u^{n+1}}\leq \norm{(1+ \Delta t \mathbb{D}^{\alpha}D)u^{n+1}} = \norm{\bar u^n -\Delta t \bar u^n Du^n}.
\end{equation*}

\end{remark}
\begin{remark}
In the discrete scheme \eqref{FDscheme}, the discretization of the convective term is similar to Holden et al. \cite{holden2015convergence}. However, the main distinction with \cite{holden2015convergence} being the inclusion of the discretized fractional term. Consequently, in the following analysis wherever the fractional term does not play a role, we refer to the approach described in \cite{holden2015convergence}.
\end{remark}

\begin{remark}
The aforementioned scheme aligns with the operator splitting scheme developed in \cite{dutta2021operator} for the fractional KdV equation \eqref{fkdv} and in \cite{holden1999operator,holden2011operator} for the KdV and generalized KdV equation. The scheme can be decomposed as follows:
    \begin{equation*}
        u_j^{n+1/2} = \bar u_j^n - \frac{\Delta t}{4\Delta x} \left((u_{j+1}^n)^2 - (u_{j-1}^n)^2\right)
    \end{equation*}
utilizing \( \bar{u}^n_jDu_j^n = \frac{1}{2}D(u_j^n)^2 \). It is noteworthy that \(u^{n+1/2}\) solves the Lax-Friedrichs scheme applied to \(u^n\) for the nonlinear part of the KdV equation. Subsequently
    \begin{equation*}
        \frac{u^{n+1}-u^{n+1/2}}{\Delta t} = -\mathbb{D}^{\alpha}(Du^{n+1}),
    \end{equation*}
where \(u^{n+1}\) is the approximation of the implicit scheme for the linear dispersive equation with the fractional Laplacian: \(u_t -(-\Delta)^{\alpha/2}u_x=0\). If we denote these two solutions as \(S_B\) and \(S_D\) respectively, then 
    $$ u^{n+1} = (S_D \circ S_B)u^n $$ solves the implicit scheme \eqref{FDscheme}.
The proofs presented here can be adopted to demonstrate the convergence of the operator splitting method for the fractional KdV equation \eqref{fkdv}.
\end{remark}

We state and proof the stability lemma which is the main ingredient in the convergence analysis:

\begin{lemma}\label{stab}
Let $u^n$ be an approximate solution obtained by the scheme \eqref{FDscheme}. Assume that CFL condition satisfies:
\begin{equation}\label{cfl}
    \lambda\|u^0\|(\frac{1}{3} + \frac{1}{2} \lambda \|u^0\|) < \frac{1-\delta}{2}, \quad \delta \in (0,1),
    \end{equation}
    where $\lambda = \Delta t/\Delta x^{3/2}$. Then 
\begin{equation}\label{stability}
        \|u^{n+1}\|^2 + \Delta x^{3}
           \lambda^2\|\mathbb{D}^{\alpha}(Du^{n+1}_j)\|^2
           + \delta \Delta x^2\|Du^n\|^2\leq\|u^n\|^2.
    \end{equation} 
\end{lemma}

\begin{proof}
Following the approach of Holden et al. \cite{holden2015convergence} and expressing the Burgers' term as
\[
B(u) = \bar u - \Delta t \bar{u}Du = \bar u - \frac{\Delta t}{2} Du^2,
\]
we have
\begin{equation}\label{norm1}
   \|B(u)\|^2  \leq \|u\|^2 - \delta \Delta x^2\|Du\|^2
\end{equation}   
provided the CFL condition \eqref{cfl} holds.
    Next, we examine the scheme \eqref{FDscheme} and it can be represented as 
    \begin{equation*}
       u^{n+1} = B(u^n) -  \Delta t \mathbb{D}^{\alpha}(Du^{n+1}).
   \end{equation*}
   Taking the $\ell^2$-norm, we obtain
   \begin{align}
          \nonumber \|B(u^n)\|^2 &= \|u^{n+1}\|^2 + 2\Delta t(u^{n+1}, \mathbb{D}^{\alpha}(Du^{n+1})) + \Delta t^2\|\mathbb{D}^{\alpha}(Du^{n+1})\|^2 \\ 
          \nonumber &=  \|u^{n+1}\|^2  + \Delta t^2\|  \mathbb{D}^{\alpha}(Du^{n+1})\|^2\\
          \label{norm2} &=  \|u^{n+1}\|^2 + \Delta x^{3}
           \lambda^2\|\mathbb{D}^{\alpha}(Du^{n+1})\|^2.
     \end{align}
  Therefore, estimates \eqref{norm1} and \eqref{norm2} imply 
\begin{equation*}
    \|u^{n+1}\|^2 + \Delta x^{3}
       \lambda^2\|\mathbb{D}^{\alpha}(Du^{n+1})\|^2
       + \delta \Delta x^2\|Du^n\|^2 \leq \|u^n\|^2.
\end{equation*}
Hence the stability of \eqref{FDscheme} is ensured.
\end{proof}
Subsequently, we explore the temporal derivative bound of the scheme \eqref{FDscheme}. This bound is significant in the forthcoming convergence proof, as the proof relies on the compactness theorem. We start with by introducing the following notation for a given function $v$
\begin{align*}
    D_+^t v(t) = \frac{1}{\Delta t}(v(t+\Delta t) - v(t)).
\end{align*}
\begin{lemma}\label{TEMP_BOUNDimp}
    Let $u^{n}$ be an approximate solution obtained by the scheme \eqref{FDscheme}. Assume that $\lambda$ from Lemma \ref{stab} satisfies:
    \begin{equation}\label{cfl2}
        6\|u_0\|^2\lambda^2 + \|u_0\|\lambda < \frac{1-\Tilde{\delta}}{2}, \quad \tilde\delta \in (0,1).
    \end{equation}
    Then there holds
    \begin{equation}\label{stability2}
        \|D_+^tu^{n}\|^2+\Delta t^2\|\mathbb{D}^{\alpha}D(D_+^tu^{n})\|^2 + \Tilde{\delta}\Delta x^2\|D(D_+^tu^{n-1})\|^2 \leq (1+3\Delta t\|Du^n\|_\infty)\|D_+^tu^{n-1}\|^2.
    \end{equation}
    Moreover, the following estimates hold:
    \begin{align}
        \label{Temp_DER}\|D_+^tu^{n}\|&\leq C,\\
        \label{H_alpha}\|u^n\|_{h^{1+\alpha}} &\leq C,
    \end{align}
    where $C$ is a constant independent of $\Delta x$.
\end{lemma}

\begin{proof}
From the scheme \eqref{FDscheme}, we have
\begin{align}\label{est1}
    D_+^t u_j^{n} = D_+^t\bar u_j^{n-1} - G(u^n)_j- \Delta t \mathbb{D}^{\alpha}(DD_+^tu^{n}_j),
\end{align}
where $G(u^n) = \bar u^n Du^n -\bar u^{n-1} Du^{n-1}$, and it can be further simplified as
\begin{align*}
    G(u^n) &= \bar u^n Du^n -\bar u^{n-1} Du^{n-1} \\
           &= \Delta t(D_+^t\bar u^{n-1} Du^n +\bar u^{n-1} DD_+^tu^{n-1})\\
           &= \Delta t(D_+^t\bar u^{n-1} Du^n +\bar u^{n} DD_+^tu^{n-1}-\Delta t            D_+^t\bar u^{n-1} DD_+^tu^{n-1})\\
           &=  \Delta t\left(D(u^nD_+^t u^{n-1})-\frac{\Delta t}{2}D(D_+^tu^{n-1})^2\right).
\end{align*}
With the help of the above identity in the equation \eqref{est1} and setting $\tau^n := D_+^tu^{n-1}$ yield
\begin{equation}\label{tau2}
    \tau^{n+1} = \sigma^n - \Delta t \mathbb{D}^{\alpha}D\tau^{n+1},
\end{equation}
where $\sigma$ is defined by
\begin{equation}\label{sigma}
    \sigma = \bar\tau - \Delta t D(u\tau)+ \frac{\Delta t^2}{2}D\tau^2.
\end{equation}

Now we will follow the same strategy applied in \cite[Lemma 3.2]{holden2015convergence} to evaluate \eqref{sigma}, which gives the following bound
\begin{equation}\label{sigma2}
    \frac{1}{2}\|\sigma^n\|^2+\Tilde{\delta}\frac{\Delta x^2}{2}\|D\tau^n\|^2\leq\frac{1}{2}\|\tau^n\|^2+\frac{3-\Tilde{\delta}}{2}\Delta t\|Du^n\|_\infty\|\tau^n\|^2.
\end{equation}
Squaring equation \eqref{tau2} and summing the resulting equation over $j\in \Z$ yields
\begin{equation*}
        \|\sigma^n\|^2 = \|\tau^{n+1}\|^2+\Delta t^2\|\mathbb{D}^{\alpha}D\tau^{n+1}\|^2.
\end{equation*}
Therefore, substituting the above identity in \eqref{sigma2}, we have
\begin{equation}\label{timebound}
    \begin{split}
        \|\tau^{n+1}\|^2+\Delta t^2\|\mathbb{D}^{\alpha}D\tau^{n+1}\|^2 + \Tilde{\delta}\Delta x^2\|D\tau^n\|^2 \leq (1+3\Delta t\|Du^n\|_\infty)\|\tau^n\|^2.
    \end{split}
\end{equation}
This gives the estimate \eqref{stability2}. Dropping the positive term from left hand side in  \eqref{timebound}, we have
\begin{equation*}
    \begin{split}
        \|\tau^{n+1}\|^2 \leq \|\tau^n\|^2+3\Delta t\|Du^n\|_\infty\|\tau^n\|^2.
    \end{split}
\end{equation*}
Again, following the approach of Holden et al. \cite{holden2015convergence} ensures the existence of $T>0$ such that the following estimate hold:
$$\|\tau^n\|\leq C, \qquad(n+1)\Delta t\leq T,$$ where $C$ is a constant independent of $\Delta x$. This is a temporal derivative bound.

Finally, utilizing these bounds, we obtain
 \begin{equation*}
        \|\mathbb{D}^{\alpha}Du^{n}\|\leq\|D_{+}^{t}u^n\| + \|\bar u^n\|_\infty\|Du^n\|\leq C,\quad  (n+1)\Delta t\leq T,
    \end{equation*}
     where $C$ is a constant independent of $\Delta x$. This implies \eqref{H_alpha}.
\end{proof}
\subsection{Convergence}
We follow the approach outlined by Sj\"{o}berg \cite{sjoberg1970korteweg} to establish the convergence of the scheme for $t < T$. The construction of the approximate solution $u_{\Delta x}$ is carried out in two distinct steps of the piece-wise interpolation. Firstly, we perform interpolation in space for each $t_n$:
\begin{equation}\label{interpolation1}
\begin{split}
    u^{n}(x) = u_{j}^{n} + Du_{j}^{n}(x-x_j) ,\quad x\in [x_j,x_{j+1}), \quad j\in\mathbb{Z}.
\end{split} 
\end{equation}
Following this, we perform interpolation in time for all $x\in\mathbb{R}$:
\begin{equation}\label{Interpolation2}
u_{\Delta x} (x,t) = u^n(x)  + D_{+}^{t}u^n(x)(t-t^n),  \quad t\in[t_n,t_{n+1}),\quad (n+1)\Delta t \leq \bar{T}.
\end{equation}
Note that the interpolation satisfies at nodes, i.e., for all $j\in \mathbb{Z}$ and $n\in \mathbb{N}_0$,  $u_{\Delta x}(x_j,t_n) = u_{j}^{n}$.

With these interpolations in place, we proceed to state and prove the main result of this section.
\begin{theorem}
    Let $u_0\in H^{1+\alpha}(\mathbb{R})$, $\alpha\in[1,2)$. Then there is a finite time $T>0$, depending on $\|u_0\|_{H^{1+\alpha}(\R)}$, such that for $t \leq T$ and $\Delta t = \mathcal{O}(\Delta x^{3/2})$, the sequence of approximate solutions obtained by the scheme \eqref{FDscheme} uniformly converges to the unique solution of the fractional KdV equation \eqref{fkdv} in $C(\mathbb{R}\times [0, T])$ as $\Delta x \xrightarrow[]{} 0$. 
\end{theorem}
   \begin{proof}
   Interpolation \eqref{Interpolation2} gives that $u_{\Delta x}$ is smooth enough. Differentiating $u_{\Delta x}$ in both space and time for $x\in[x_j,x_{j+1})$ and $t\in[t_n,t_{n+1})$ gives
\begin{equation*}
    \begin{split}
    \partial_x u_{\Delta x}(x,t) &= Du_j^n + D_+^t\left(Du_j^n\right)(t-t_n),\\
    \partial_t u_{\Delta x}(x,t) &= D^+_tu^n(x),
    \end{split}
\end{equation*}
which clearly implies for all $t\leq T$,
 \begin{align}
        &\|u_{\Delta x}(\cdot,t)\|_{L^2(\mathbb{R})}  \leq \|u_0\|_{L^2(\mathbb{R})} \label{bound1},\\
        &\|\partial_x u_{\Delta x}(\cdot,t)\|_{L^2(\mathbb{R})} \leq C \label{bound2},\\
    & \|\partial_t u_{\Delta x}(\cdot,t)\|_{L^2(\mathbb{R})} \leq C \label{bound3},\\
    & \|-(-\Delta)^{\alpha/2}\partial_{x} u_{\Delta x}(\cdot,t)\|_{L^2(\mathbb{R})} \leq C \label{bound4},
\end{align}
 where $C$ is a constant independent of $\Delta x$. The first estimate \eqref{bound1} follows from the exact integration of the square of \eqref{Interpolation2} over each interval $[x_j,x_{j+1})$ and summation over $j$. Similarly, estimate \eqref{stability} implies \eqref{bound2} and \eqref{Temp_DER} implies \eqref{bound3}. Since for $t<T$, $\partial_x u_{\Delta x}(\cdot,t)$ is constant in the interval $[x_j,x_{j+1})$ for all $j\in \mathbb Z$, we can take into account the Lemma \ref{D_C} and bound \eqref{H_alpha} from the Lemma \ref{TEMP_BOUNDimp} to establish the estimate \eqref{bound4} as $\Delta x \to 0$.

The temporal derivative bound on the approximate solutions $u_{\Delta x}$ establishes that for every possible $\Delta x>0$, $u_{\Delta x}\in \text{Lip}([0,T];L^2(\mathbb{R}))$. Employing the bound \eqref{bound1}, we can apply the Arzelà-Ascoli theorem, indicating that the set of approximate solutions $\{u_{\Delta x_j}\}_{j\in\mathbb Z}$ is sequentially compact in $C([0,T];L^2(\mathbb{R}))$. Consequently, this implies the existence of a subsequence $\Delta x_{j_{_{k}}}$ such that 
\begin{equation}\label{CONV_Utouimp}
    u_{\Delta x_{j_{_{k}}}} \xrightarrow[]{\Delta x_j \xrightarrow[]{}0} u \text{ uniformly in } C([0,T];L^2(\mathbb{R})).
\end{equation}
   Now we claim that $u$ is a weak solution of the equation \eqref{fkdv}, that is, we need show that $u$ satisfies the following equation:
\begin{equation}\label{weaks_Odximp}
    \int_0^T \int_{-\infty}^{\infty}\left(\varphi_t u + \varphi_x \frac{u^2}{2} -(-\Delta)^{\alpha/2}\varphi_x u\right)\,dx\,dt + \int_{-\infty}^{\infty}\varphi(x,0)u_0(x)\,dx = \mathcal{O}(\Delta x),
\end{equation}
for all test functions $\varphi\in C_c^\infty(\R\times[0,T])$.

We employ a Lax-Wendroff type argument inspired by \cite{holden1999operator}. Let us define a piecewise constant interpolation for the approximate solution by the following:
\begin{equation}\label{simpintp}
    \bar{u}_{\Delta x}(x,t) = u_j^n  \text{ for } x\in [x_j,x_{j+1}),~j\in\Z \text{ and } t\in[t_n,t_{n+1}), ~ t_{n+1}\leq T.
\end{equation}
Assuming $\Delta x$ is sufficiently small, it is convenient to use a interpolation \eqref{simpintp} instead of interpolation \eqref{Interpolation2} to follow the proof of Lax-Wendroff type result. It is worth noting that $\bar{u}_{\Delta x}(\cdot,t_n)\to u(\cdot,t_n)$ in $L^2(\mathbb{R})$ as $\Delta x \to0$ for every $t_n\leq T$.  
Let us take a test function $\varphi\in C_c^\infty(\R\times[0,T])$ and denote $\varphi(x_j,t_n) = \varphi_j^n$ at nodes. We multiply \eqref{FDscheme} by $\Delta x \Delta t \varphi_j^n $ and taking the summation over all $n$ and $j$ to obtain 
    \begin{align*}
    \nonumber\Delta t \Delta x \sum_n\sum_j\varphi_j^n \left(\frac{u_j^{n+1}-\bar u_j^{n}}{\Delta t}\right) +& \Delta t \Delta x \sum_n\sum_j\varphi_j^n \frac{D(u^n_j)^2}{2}\\ \label{Sumscheme} +&\Delta t \Delta x \sum_n\sum_j\varphi_j^n \mathbb{D}^{\alpha} Du_j^{n+1/2} = 0, \quad n\Delta t\leq T, \quad j\in \mathbb{Z}.
\end{align*}
Following the approach in \cite{dutta2016convergence} and \cite{holden1999operator}, we can show that 
    \begin{align*}
        \Delta t \Delta x \sum_n\sum_j\varphi_j^n \left(\frac{u_j^{n+1}-\bar u_j^{n}}{\Delta t}\right) &= -\Delta t \Delta x \sum_n\sum_j D_+^t\varphi_j^n u^{n}_j\\& \xrightarrow[]{\Delta x \xrightarrow[]{} 0} -\int_0^T\int_\R u\varphi_t\,dx\,dt -  \int_{\R}\varphi(x,0)u_0(x)\,dx,
    \end{align*}
    and 
     \begin{align*}
        \Delta t \Delta x\sum_n\sum_j\varphi_j^n \frac{D(u^n_j)^2}{2} \xrightarrow[]{\Delta x \xrightarrow[]{} 0} -\int_0^T\int_\R \frac{u^2}{2}\varphi_x\,dx\,dt.
    \end{align*}
Now we estimate the term involving the fractional Laplacian by using the properties of discrete fractional Laplacian described in Lemma \ref{P_fracLap}.
   Hereby we use the same notation for the inner product in $L^2$ and $\ell^2$.
   \begin{align*}
       \Delta t \Delta x \sum_n\sum_j\varphi_j^n \mathbb{D}^{\alpha} Du_j^{n+1} =& -\Delta t \sum_n\langle u^{n+1}, \mathbb{D}^{\alpha} D\varphi^n\rangle.
   \end{align*}
Since $\varphi(\cdot,t_n)$ and $\varphi_x(\cdot,t_n)$ are smooth functions, $D\varphi^n$ converges to $\varphi_x(\cdot,t_n)$ uniformly as $\Delta x\to 0$. Then, we have
    \begin{align*}
    \Biggr|\langle  u^{n+1}, &\mathbb{D}^{\alpha} D\varphi^n\rangle - \langle  u(\cdot,t_{n+1}),-(-\Delta)^{\alpha/2} \varphi_x(\cdot,t_n)\rangle\Biggr|\\
    \leq& \Biggr|\langle  u^{n+1}-u(\cdot,t_{n+1}), \mathbb{D}^{\alpha} D\varphi^n\rangle\Biggr|+\Biggr|\langle  u(\cdot,t_{n+1}), \mathbb{D}^{\alpha} D\varphi^n - (-(-\Delta)^{\alpha/2}) \varphi_x(\cdot,t_n) \rangle\Biggr|\\
    \leq & \norm{u^{n+1}-u(\cdot,t_{n+1})}\norm{\mathbb{D}^{\alpha} D\varphi^n} + \norm{u(\cdot,t_{n+1})}\norm{\mathbb{D}^{\alpha}D\varphi^n - (-(-\Delta)^{\alpha/2})\varphi_x(\cdot,t_n)}.
    \end{align*}
The first term converges to zero using \eqref{CONV_Utouimp}, and by applying the Lemma \ref{consit}, it is evident that the second term also vanishes as $\Delta x\to0$. Consequently, we have demonstrated that $u$ satisfies \eqref{weaks_Odximp}, which signifies that it is a weak solution.
     
Finally, the estimates \eqref{bound1}-\eqref{bound4} ensure that the weak solution $u$ satisfies the equation \eqref{fkdv} as an $L^2$-identity. Hence considering the initial data $u_0\in H^{1+\alpha}(\R)$, the limit $u$ becomes the unique solution of the fractional KdV equation \eqref{fkdv}. This completes the proof.
\end{proof}
\section{Crank-Nicolson finite difference scheme}\label{sec4}
The essence of this study lies in the unveiling of a robust Crank-Nicolson temporal discretized finite difference scheme tailored for the precise numerical approximation of the solution to the fractional KdV equation \eqref{fkdv}, encapsulated by the following expression:
\begin{equation}\label{CNscheme}
    u_j^{n+1} = u_j^{n} - \Delta t \mathbb G(u^{n+1/2}) - \Delta t \mathbb{D}^{\alpha} Du_j^{n+1/2}, \quad n\in\mathbb{N}_0, \quad j\in \mathbb{Z},
\end{equation}
where $\mathbb G(u^{n+1/2}) := \Tilde{u}^{n+1/2}Du^{n+1/2}$ and $ u^{n+1/2}:= \frac{1}{2}(u^n + u^{n+1})$. For the discretiation of the initial data, we set $u_j^0 = u_0(x_j)$ for $j\in\mathbb{Z}$. In order to ensure that the scheme is well-defined and guarantee the existence of a solution, we strategically adopt the proven methodology explained in \cite{dutta2016convergence}, involving a fixed-point iteration approach.
For the solvability of scheme \eqref{CNscheme}, we introduce the sequence $\{w^\ell\}_{\ell\geq 0}$ for the fixed-point iteration with $w^{\ell+1}$ as the solution to the following equation:
\begin{equation}\label{iterr}
    \begin{cases}
        w^{\ell+1} = u^n - \Delta t \mathbb{G}\left(\frac{u^n+w^\ell}{2}\right) - \Delta t \mathbb{D}^\alpha D\left(\frac{u^n+w^{\ell+1}}{2}\right),\\
        w^0 = u^n.
    \end{cases}
\end{equation}
To establish the existence and uniqueness of the sequence $w^\ell$, we reformulate the iteration in a linear framework:
\begin{equation}\label{Pos_Def}
    \left(1+\frac{\Delta t}{2}\mathbb{D}^\alpha D\right)w^{\ell+1} =  u^n - \Delta t \mathbb{G}\left(\frac{u^n+w^\ell}{2}\right) - \frac{\Delta t}{2} \mathbb{D}^\alpha Du^n.
\end{equation}
A key insight, supported by the Lemma \ref{P_fracLap}, affirms that the operator $\frac{\Delta t}{2}\mathbb{D}^\alpha D$ is skew-symmetric. This property renders the coefficient operator on the left-hand side of \eqref{Pos_Def} to be positive definite, thereby ensuring the existence and uniqueness of the iterative sequence \eqref{iterr}. 

We describe the following lemma which establishes that the scheme \eqref{CNscheme} is solvable at each time step. In addition,
the lemma serves as a cornerstone for our subsequent stability analysis.
\begin{lemma}\label{wllemma}
    Let $K = \frac{6-L}{1-L} > 6$ be a constant and $L\in(0,1)$. Consider the fixed-point iteration defined by \eqref{iterr}. Assume that the CFL condition for $\Delta x$ and $\Delta t$:
    \begin{equation}\label{CNCFL}
        \lambda \leq \frac{L}{K\|u^n\|_{h^2}},
    \end{equation}
    where $\lambda = \frac{\Delta t}{\Delta x}$. Then there exists a solution $u^{n+1}$ to equation \eqref{CNscheme} and $\lim_{\ell\to\infty} w^{\ell} = u^{n+1}$. Moreover, the following estimate holds:
    \begin{equation}\label{stabn}
        \|u^{n+1}\|_{h^2} \leq K \|u^n\|_{h^2}.
    \end{equation}
\end{lemma} 
\begin{proof}
Set $\Delta w^\ell := w^{\ell+1} - w^\ell $, we have
\begin{equation}\label{iterrwl}
    \left(1+\frac{1}{2}\Delta t \mathbb{D}^{\alpha}D\right)\Delta w^\ell = -\Delta t \left[\mathbb{G}\left(\frac{u^n+w^\ell}{2}\right) - \mathbb{G}\left(\frac{u^n+w_{\ell-1}}{2}\right)\right]:=-\Delta t \Delta \mathbb{G}.
\end{equation}
Applying the discrete operator $D_+D_-$ to \eqref{iterrwl}, then multiply with $\Delta x D_+D_-\Delta w^\ell$ and summing over $j\in \mathbb{Z}$, we have
\begin{equation*}
   \begin{split}
    \|D_+D_-\Delta w^\ell\|^2 &= -\Delta t \left\langle D_+D_-\Delta \mathbb{G}, D_+D_-\Delta w^\ell\right\rangle\\
     &\leq  \Delta t  \|D_+D_-\Delta \mathbb{G}\| \|D_+D_-\Delta w^\ell\|.
     \end{split}
\end{equation*}
where we have used the Lemma \ref{P_fracLap} for the fractional term. 
Following closely the steps from Dutta et al. \cite[Lemma 2.5]{dutta2016convergence}, we have that the sequence $\{w^\ell\}$ is Cauchy, hence converges. In addition, we have the estimate \eqref{stabn}. 
\end{proof}

\begin{remark}
    We have shown that the devised scheme is solvable for each time step assuming the CFL condition \eqref{CNCFL}, where $\lambda$ is bounded by the $h^2$-bound of approximate solution $u^n$. Since the CFL condition \eqref{CNCFL} depends on $u^n$, not on the initial condition directly. Consequently, in order to provide a comprehensive assessment of the stability of our computed solution $u^n$, we must embark on a thorough stability analysis that delves into the intricacies of the evolving solution over time.
\end{remark}

Now we prove a fundamental stability lemma.
\subsection{Stability Lemma}
\begin{lemma}\label{Stablemma}
    Let $u_0\in H^{1+\alpha}(\R)$. Assume that $\Delta t$ satisfies 
    \begin{equation}\label{CFLCNi}
        \lambda \leq \frac{L}{KY}
    \end{equation}
    for some $Y=Y\left(\|u_0\|_{H^2(\R)},\|u_0\|_{H^{1+\alpha}(\R)}\right)$ and $\lambda = \Delta t/\Delta x$. Then there is a finite time $T>0$ and a constant $C$ both depending on $\|u_0\|_{h^2}$ and $\|u_0\|_{h^{1+\alpha}}$ such that 
    \begin{align}
        \label{stabcn}\|u^n\|_{h^2} &\leq C, \qquad \text{ for }  t_n\leq T, \\
        \label{tempbound} \|D_+^t u^n\| &\leq C, \qquad \text{ for }  t_n\leq T,\\
         \label{stabcn2}\|u^n\|_{h^{1+\alpha}} &\leq C, \qquad \text{ for }  t_n\leq T.
    \end{align}
    \eqref{stabcn2} is a stability estimate.
\end{lemma}
\begin{proof}
Motivated by the analysis established in \cite[Lemma 2.7]{dutta2016convergence}, we perform the difference operator $D_+D_-$ on the scheme \eqref{CNscheme} and taking the inner product with $D_+D_- u^{n+1/2}$ amounts to
 \begin{equation}
   \|D_+D_-u^{n+1}\|^2 =\|D_+D_-u^n\|^2 - 2\Delta t \langle D_+D_-\mathbb{G}(u^{n+1/2}), D_+D_-u^{n+1/2}\rangle,
   \end{equation}
where we have used  Lemma \ref{P_fracLap} for the fractional term. We estimate the nonlinear part on the right hand side by following steps from \cite[Lemma 2.7]{dutta2016convergence}, which yields 
\begin{equation*}
    |\langle D_+D_-\mathbb{G}(u^{n+1/2}), D_+D_-u^{n+1/2}\rangle| \leq \sqrt{\frac{3}{2}}\norm{D_+D_-u^{n+1/2}} \norm{u^{n+1/2}}^2_{h^2}. 
\end{equation*}
We can get similar estimate for the lower order derivative term. This further gives the following bound 
\begin{equation}\label{timeb}
   \norm{u^{n+1}}_{h^2} \leq  \norm{u^n}_{h^2} +  8\Delta t \norm{u^{n+1/2}}^2_{h^2}.
\end{equation}
Let us introduce a differential equation $$y'(t) = 2(K+1)^2y(t)^2,\quad t>0; \qquad y(0) = \max\{\|u_0\|_{H^{2}},\|u_0\|_{H^{1+\alpha}}\}=:y_0.$$ It is observed that the solution $y(t)$ of the above differential equation is convex and increasing for all $t<T:=T_\infty/2$, where  ${T}_\infty = 1/(2(K+1)^2y_0) $.
     
Next we claim that 
\begin{align}\label{stab:temp_1}
\|u^n\|_{h^2}\leq y(t_n)\leq Y \quad \text{for } t_n\leq T.
\end{align}
We proceed by mathematical induction. The claim is obvious for $n=0$ by the Lemma \ref{D_C}. Now we assume that the estimate holds for $n=1,\dots,m.$ As $\|u^m\|_{h^2}\leq y(t_m) \leq y(T) =:Y(y_0)$, then the CFL condition \eqref{CFLCNi} implies \eqref{CNCFL}. Thus by the Lemma \ref{wllemma}, we have 
     \begin{equation}\label{n+1bound}
         \|u^{m+1/2}\|_{h^2} \leq \frac{(K+1)}{2}\|u^m\|_{h^2}.
     \end{equation}
Since $y$ is increasing and convex, then \eqref{timeb} and \eqref{n+1bound} yield the following estimate
    \begin{equation*}
        \begin{split}
            \|u^{m+1}\|_{h^2} \leq \|u^{m}\|_{h^2} +2\Delta t((K+1)\|u^{m}\|_{h^2})^2 \leq &  y(t_m) +2\Delta t((K+1)y(t_m))^2\\
            \leq & y(t_m) +\int_{t_m}^{t_{m+1}}2(K+1)^2y(t_m)^2\,dt\\
            \leq &   y(t_m) +\int_{t_m}^{t_{m+1}}y'(s)\,ds = y(t_{m+1}).
        \end{split}
    \end{equation*}
This proves that $\|u^n\|_{h^2} \leq y(T) = Y$, $(n+1)\Delta t<T$. Hence the estimate \eqref{stab:temp_1} holds. 

From the scheme \eqref{CNscheme}, we have
    \begin{equation*}
         u_j^{n+1} = u_j^{n} - \Delta t \mathbb{G}(u^{n+1/2}) - \Delta t \mathbb{D}^{\alpha} Du_j^{n+1/2}
    \end{equation*}
    and
    \begin{equation*}
         u_j^{n} = u_j^{n-1} - \Delta t \mathbb{G}(u^{n-1/2}) - \Delta t \mathbb{D}^{\alpha} Du_j^{n-1/2},
    \end{equation*}
    which implies
    \begin{equation*}
         D_+^tu_j^{n} = D_+^tu_j^{n-1} - \left(\mathbb{G}(u^{n+1/2}) -\mathbb{G}(u^{n-1/2})\right) - \Delta t \mathbb{D}^{\alpha} D (D_+^tu_j^{n-1/2}).
    \end{equation*}
    Taking inner product with $D_+^tu_j^{n-1/2} = \frac{1}{2}(D_+^tu_j^{n} + D_+^tu_j^{n-1})$, we obtain
    \begin{equation*}
    \begin{split}
       \frac{1}{2} \| D_+^tu^{n}\|^2 - & \frac{1}{2} \| D_+^tu^{n-1}\|^2 = -\left\langle \left(\mathbb{G}(u^{n+1/2}) -\mathbb{G}(u^{n-1/2})\right), D_+^tu^{n-1/2}\right\rangle\\
       = & -\left\langle \left(\Tilde u^{n+1/2}Du^{n+1/2} - \Tilde u^{n-1/2}Du^{n-1/2}\right), D_+^tu^{n-1/2}\right\rangle\\
       = & -\Delta t\left\langle \left(D_+^t\Tilde u^{n-1/2}Du^{n+1/2} + \Tilde u^{n-1/2}DD_+^tu^{n-1/2}\right), D_+^tu^{n-1/2}\right\rangle\\
       \leq & \Delta t \norm{Du^{n+1/2}}_{\infty}\norm{D_+^tu^{n-1/2}}^2 + \Delta t \left\langle D(\Tilde u^{n-1/2}D_+^tu^{n-1/2}), D_+^tu^{n-1/2}\right\rangle\\
       \leq & \Delta t \norm{Du^{n+1/2}}_{h^2}\norm{D_+^tu^{n-1/2}}^2 + \Delta t\Biggr[\frac{\Delta x}{2} \left\langle D_+\Tilde u^{n-1/2}DD_+^tu^{n-1/2}, D_+^tu^{n-1/2}\right\rangle \\& \qquad\qquad\qquad \qquad \qquad \qquad\qquad \qquad + \frac{1}{2}\left\langle S^-D_+^tu^{n-1/2}D\Tilde u^{n-1/2} ,D_+^tu^{n-1/2} \right\rangle\Biggr]\\
       \leq &C \Delta t  \left(\norm{D_+^tu^{n}}^2 + \norm{D_+^tu^{n-1}}^2\right), 
    \end{split}
    \end{equation*}
    where we have used \eqref{Dz_1z_2} and  the discrete Sobolev inequality $\|Du^n\|_{\infty}\leq \|u^n\|_{h^2} \leq C$. Assuming $\Delta t$ is small enough such that $1-C\Delta t \geq \frac{1}{2}$, then we have 
    \begin{align*}
          \| D_+^tu^{n}\|^2  \leq  \| D_+^tu^{n-1}\|^2 +2 \Delta t \norm{D_+^tu^{n-1}}^2.
    \end{align*}
    By setting $\Gamma_n = \| D_+^tu^{n-1}\|^2$ for every $n\Delta t \leq T$, we see that 
    \begin{equation*}
         \Gamma_{n+1} \leq \Gamma_n + 2 C \Delta t \Gamma_n.
     \end{equation*}
Let $A(t)$ solves the differential equation 
    $$A'(t) =  2CA(t), \quad A(t_1) = \Gamma_1.$$
Note that $\Gamma_1$ is finite as $\|u^1\|$ and $\|u^0\|$ are bounded by \eqref{stabcn} and there exist a sufficiently large $\bar T$ such that the solution $A(t)$ is bounded for every $t<T<\bar T$. Clearly, A is increasing and convex, then $A(t_n)\leq A(T)$ for every $t_n\leq T$. We claim that $\Gamma_n\leq A(t_n)$ for $t_n\leq T$. We use the mathematical induction, as this holds for $n=1$ by construction. Let  $\Gamma_n\leq A(t_n)$ for $n=2,\dots,N$, then
     \begin{equation*}
         \begin{split}
             \Gamma_{N+1} \leq \Gamma_N + 2 C \Delta t \Gamma_N  \leq & A(t_N) + 2C\Delta tA(t_N) \\
             = & A(t_N) + \int_{t_N}^{t_{N+1}} 2 C A(t_N)\,dt\\
             \leq & A(t_N) + \int_{t_N}^{t_{N+1}} A'(s)\,ds = A(t_{N+1}),
             \end{split}
     \end{equation*}
where we have used that $A$ is increasing. Hence for $t_n\leq T$, $\Gamma_{n+1} = \| D_+^tu^{n}\|^2\leq A(T)=:C$, where $C$ is a constant independent of $\Delta x$. Finally we end up with
     \begin{equation*}
        \|\mathbb{D}^{\alpha}Du^{n+1/2}\|\leq \|D_+^tu^n\| + \|u^n\|_{\infty}\|Du^n\|\leq C, \qquad t_n\leq T.
    \end{equation*}
   This implies $ \|u^n\|_{h^{1+\alpha}} \leq C $ for $t_n\leq T.$
   Hence the result follows.
\end{proof}
\subsection{Convergence of approximate solution}
We adopt the approach outlined by Sjoberg \cite{sjoberg1970korteweg} to establish the convergence of the scheme \eqref{CNscheme} for $t<T$. Our reasoning unfolds through the construction of a piecewise continuous interpolation, denoted as $u_{\Delta x}$, defined by \eqref{Interpolation2}.
With this, we are poised to present the main result of this section.
\begin{theorem}
Assume that $\|u_0\|_{h^{1+\alpha}},~\alpha\in[1,2)$, is finite. Let $\{u_{\Delta x}\}_{\Delta x>0}$ be a sequence of approximate solutions obtained by the scheme \eqref{CNscheme} of the fractional KdV equation \eqref{fkdv}. Then there exists a finite time $T>0$ and a constant $C$, depending on $\|u_0\|_{h^{1+\alpha}}$  such that 
    \begin{align}
        \label{bb1}\|u_{\Delta x}(\cdot,t)\|_{L^2(\mathbb{R})} &\leq \|u_0\|_{L^2(\mathbb{R})},\\
       \label{bb2} \|\partial_xu_{\Delta x}(\cdot,t)\|_{L^2(\mathbb{R})} &\leq C,\\
         \label{bb3} \|\partial_tu_{\Delta x}(\cdot,t)\|_{L^2(\mathbb{R})} &\leq C,\\
          \label{bb4} \left\|(-\Delta)^{\alpha/2}\partial_xu_{\Delta x}(\cdot,t)\right\|_{L^2(\mathbb{R})} &\leq C,
    \end{align}
for $\Delta t = \mathcal{O}(\Delta x)$. Moreover, the sequence of approximate solutions $\{u_{\Delta x}\}_{\Delta x>0}$ converge uniformly to the unique solution of the fractional KdV equation \eqref{fkdv} in $C(\mathbb{R}\times [0,T]) $  as $\Delta x\to 0$.
\end{theorem}
\begin{proof}
Since the interpolation function $u_{\Delta x}$, defined by \eqref{Interpolation2}, is smooth, we can explicitly express its spatial and temporal derivatives:
    \begin{align}
        \label{int2}\partial_x u_{\Delta x}(x,t) =& Du_j^n + (t-t_n)D_+^t\left(Du_j^n\right),\\
        \label{int3}\partial_t u_{\Delta x}(x,t) =& D_+^tu^n(x).
    \end{align}
The above expressions imply the estimates \eqref{bb1}, \eqref{bb2} and \eqref{bb3} for $t \leq T$ by integrating over $[x_j,x_{j+1})$ the square of \eqref{Interpolation2}, \eqref{int2} and \eqref{int3}, respectively, followed by the summation over $j$. Since for $t<T$, $\partial_x u_{\Delta x}(\cdot,t)$ is constant in the interval $[x_j,x_{j+1})$ for all $j\in \mathbb Z$, we can apply the Lemma \ref{D_C} and the estimate \eqref{stabcn2} from Lemma \ref{Stablemma} to establish the validity of \eqref{bb4} as $\Delta x \xrightarrow[]{}0$.
    
The temporal derivative bound on the approximate solutions $u_{\Delta x}$ establishes that $u_{\Delta x}\in \text{Lip}([0,T];L^2(\mathbb{R}))$ for every possible $\Delta x>0$. Employing the bound \eqref{bb1}, we can apply the Arzelà-Ascoli theorem, which implies that the set of approximate solutions $\{u_{\Delta x_j}\}_{j\in\mathbb Z}$ is sequentially compact in $C([0,T];L^2(\mathbb{R}))$. Consequently, this implies the existence of a subsequence $\Delta x_{j_{_{k}}}$ such that 
\begin{equation}\label{CONV_Utou}
    u_{\Delta x_{j_{_{k}}}} \xrightarrow[]{\Delta x_j\xrightarrow[]{} 0} u \text{ uniformly in } C([0,T];L^2(\mathbb{R})).
\end{equation}
Our next aim is to demonstrate that the limit $u$ satisfies \eqref{weaks_Odximp}. For this, we proceed through a Lax-Wendroff type argument inspired by \cite{holden1999operator}. It is convenient to use simple interpolation function given by \eqref{simpintp} since $\bar{u}_{\Delta x}\xrightarrow[]{\Delta x\xrightarrow[]{} 0} u$ in $L^\infty([0,T];L^2(\mathbb{R}))$. 

We proceed similarly as in the previous section by considering a test function $\varphi\in C_c^\infty(\R\times[0,T])$ and denote $\varphi(x_j,t_n) = \varphi_j^n$ at nodes. We multiply $ \Delta t \Delta x \varphi_j^n $ to the scheme \eqref{CNscheme} and summing it over all $j$ and $n$ to obtain 
    \begin{align}
    \nonumber \Delta t \Delta x \sum_n\sum_j\varphi_j^n D^t_+u_j^{n} +&  \Delta t \Delta x \sum_n\sum_j\varphi_j^n \mathbb G(u^{n+1/2})\\ \label{Sumscheme} +& \Delta t \Delta x \sum_n\sum_j\varphi_j^n \mathbb{D}^{\alpha} Du_j^{n+1/2} = 0, \quad n\Delta t\leq T, \quad j\in \mathbb{Z}.
\end{align}
With the help of \cite[Theorem 2.8]{dutta2016convergence}, we have
    \begin{align*}
        \Delta t \Delta x \sum_n\sum_j\varphi_j^n D^t_+u_j^{n} \xrightarrow[]{} -\int_0^T\int_\R u\varphi_t\,dx\,dt -  \int_{\R}\varphi(x,0)u_0(x)\,dx \text{ as } \Delta x \xrightarrow[]{} 0,
    \end{align*}
    and 
     \begin{align*}
         \Delta t \Delta x \sum_n\sum_j\varphi_j^n \mathbb G(u^{n+1/2}) \xrightarrow[]{} -\int_0^T\int_\R \frac{u^2}{2}\varphi_x\,dx\,dt \text{ as } \Delta x \xrightarrow[]{} 0.
    \end{align*}
For the term involving the discrete fractional Laplacian, we use summation-by-parts to conclude that
    \begin{equation*}
        \Delta t \Delta x \sum_n\sum_j\varphi_j^n \mathbb{D}^{\alpha} Du_j^{n+1/2} \xrightarrow[]{}-\int_0^T\int_\R -(-\Delta)^{\alpha/2}\varphi_x u \,dx\,dt\text{ as } \Delta x \xrightarrow[]{} 0.
    \end{equation*}
Consequently, we have demonstrated that $u$ satisfies \eqref{weaks_Odximp}, signifying that $u$ is a weak solution of the fractional KdV equation \eqref{fkdv}.
Hence we can conclude from the estimates \eqref{bb1}-\eqref{bb4} that the weak solution $u$ becomes a strong solution, which satisfies equation \eqref{fkdv} as an $L^2$-identity. Thus, considering the initial data $u_0\in H^{1+\alpha}(\R)$, $u$ becomes the unique solution of the fractional KdV equation \eqref{fkdv}. This completes the proof.
\end{proof}
\begin{remark}($L^2$-conservative)\\
Beyond mere formulation, the fully discrete scheme \eqref{CNscheme} not only demonstrates its computational capability but also exhibits a remarkable $L^2$-conservativity. To justify the assertion, we perform a summation over $j$ after multiplying $\Delta x u_j^{n+1/2}$ in \eqref{CNscheme}, yielding the equation:
\begin{equation*}
    \|u^{n+1}\|^2 = \|u^n\|^2 - \Delta t \langle \mathbb{G}(u^{n+1/2}) , u^{n+1/2}\rangle - \Delta t \langle \mathbb{D}^{\alpha}D(u^{n+1/2}), u^{n+1/2}\rangle.
\end{equation*}
This further establishes the $L^2$-conservative nature of the scheme \eqref{CNscheme} $$\|u^{n+1}\|= \|u^n\|,$$ where we have incorporated the orthogonality conditions $\langle \mathbb{G}(u) , u \rangle = 0$ and $\langle \mathbb{D}^{\alpha} D(u), u \rangle = 0$.
\end{remark}
\section{Numerical experiment}\label{sec5}
In this section, we verify the schemes \eqref{FDscheme} and \eqref{CNscheme} with a series of numerical illustrations. We follow the conventional approach, which typically involve a periodic case of the initial value problem with periodic initial data. We consider a large enough domain in all the cases such that the initial data is compactly supported within the domain, for instance, kindly refer to \cite{dutta2016convergence,dwivedi2023stability,dwivedi2023convergence,holden2015convergence}.
However, in particular, our theoretical study in this paper focuses on the convergence of the approximated solution on the real line. To address this, we discretize the domain that is large enough in space for the reference solutions (exact or higher-grid solutions) to be nearly zero outside of it. Exact solutions are available for the cases $\alpha=1$ and $\alpha=2$.  

Let us denote the approximate solutions by $u_{\Delta x}^{EI}$ and $u_{\Delta x}^{CN}$ generated by the Euler implicit \eqref{FDscheme} and Crank-Nicolson \eqref{CNscheme} schemes respectively. We introduce the relative error as
 \begin{equation*}
     E_{EI}:= \frac{\|u_{\Delta x}^{EI}-u\|_{L^2}}{\|u\|_{L^2}}, \qquad E_{CN}:= \frac{\|u_{\Delta x}^{CN}-u\|_{L^2}}{\|u\|_{L^2}},
 \end{equation*}
 where $E_{EI}$ and $E_{CN}$ are the relative errors corresponding to the Euler implicit \eqref{FDscheme} and Crank-Nicolson \eqref{CNscheme} respectively. The $L^2$-norms involved above were computed at the grid points \(x_j\) by trapezoidal rule.
 
Hereby we examine the first three specific quantities—\textit{mass}, \textit{momentum} and \textit{energy} for the scheme \eqref{CNscheme}—as introduced in \cite{kenig1993cauchy}. These quantities, being normalized, are defined as follows:
\begin{align*}
    C^{\Delta}_1 &:= \frac{\int_{\mathbb{R}} u_{\Delta x}^{CN}\,dx}{\int_{\mathbb{R}} u_0\,dx},\qquad
    C^{\Delta}_2 := \frac{\|u_{\Delta x}^{CN}\|_{L^2(\R)}}{\norm{u_0}_{L^2(\R)}},\\
    C^{\Delta}_3 &:= \frac{\int_{\R} \left(((-\Delta)^{\alpha/4}u_{\Delta x}^{CN})^2 - \frac{1}{3}(u_{\Delta x}^{CN})^3\right)~dx}{\int_{\R} \left(((-\Delta)^{\alpha/4}u_0)^2 - \frac{1}{3}(u_0)^3\right)~dx}.
\end{align*}
Our objective is to preserve these quantities within our discrete framework, ensuring that $C^{\Delta}_i \xrightarrow[]{} 1$, $i=1,2,3$ as number of nodes $N$ increases.
 It is worth noting that in the domain of integro partial differential equations, maintaining a greater number of conserved quantities through numerical methods often leads to more accurate approximations compared to those preserving fewer quantities.

Moreover, we investigate the convergence rates of the numerical schemes \eqref{FDscheme} and \eqref{CNscheme}, labeled as $R_{EI}$ and $R_{CN}$ respectively. The convergence rates are computed using the expressions:
\begin{equation*}
    R_{EI} =  \frac{\ln(E_{EI}(N_1))-\ln(E_{EI}(N_2))}{\ln(N_2)-\ln(N_1)} ~~~\text{ and }~~~ R_{CN} = \frac{\ln(E_{CN}(N_1))-\ln(E_{CN}(N_2))}{\ln(N_2)-\ln(N_1)},
\end{equation*}
where $E_{EI}$ and $E_{CN}$ are treated as functions dependent on the number of nodes $N_1$ and $N_2$.
The CFL condition was imposed with the time step $\Delta t = 0.5\Delta x/\|u_0\|_{\infty}$.

The validation of our exposition will manifest in the forthcoming examples:

\subsection{Benjamin-Ono equation}{$\alpha= 1$}:\\
Let us consider the solution of the Benjamin-Ono equation as introduced in \cite{thomee1998numerical}:

\begin{equation}\label{BOsolution}
    u_1(x,t) = \frac{2c\delta}{1-\sqrt{1-\delta^2}\cos(c\delta(x-ct))}, \qquad \delta = \frac{\pi}{cL}.
\end{equation}

We implemented both the schemes using the initial data $u^0(x) = u_1(x,0)$ with parameters  $L=15$ and $c=0.25$. By choosing $\alpha=1$ in both schemes \eqref{FDscheme} and \eqref{CNscheme}, we compare the obtained approximate solutions with the reference solution given by \eqref{BOsolution} at the time $t=20$, $t=100$ and $t=120$. Since the solution is periodic, the time $t=120$, represents one period for the exact solution \eqref{BOsolution}.

A graphical representation of the results for $N=512$ is provided in Figure \ref{fig:BOoneCN} for the time $t=20$, $t=100$ and $t=120$. The plot distinctly shows that the approximate solution $u_{\Delta x}^{CN}$ closely aligns with the exact solution compared to the approximate solution $u_{\Delta x}^{EI}$. This observation is further substantiated by the errors presented in Table \ref{tab:TableBOCN} for the time $t=120$ and for the other times outputs are very similar, where the errors exhibit a decreasing trend with a rate of approximately $1$ for the Euler implicit scheme \eqref{FDscheme} and $2$ for the Crank-Nicolson scheme \eqref{CNscheme}. Furthermore, Table \ref{tab:TableBOCN} verifies that the Crank-Nicolson scheme \eqref{CNscheme} conserves the aforementioned quantities $C_i^{\Delta}$, $i=1,2,3$.

\begin{figure}
    \centering
    \includegraphics[width=0.9 \linewidth, height=10cm]{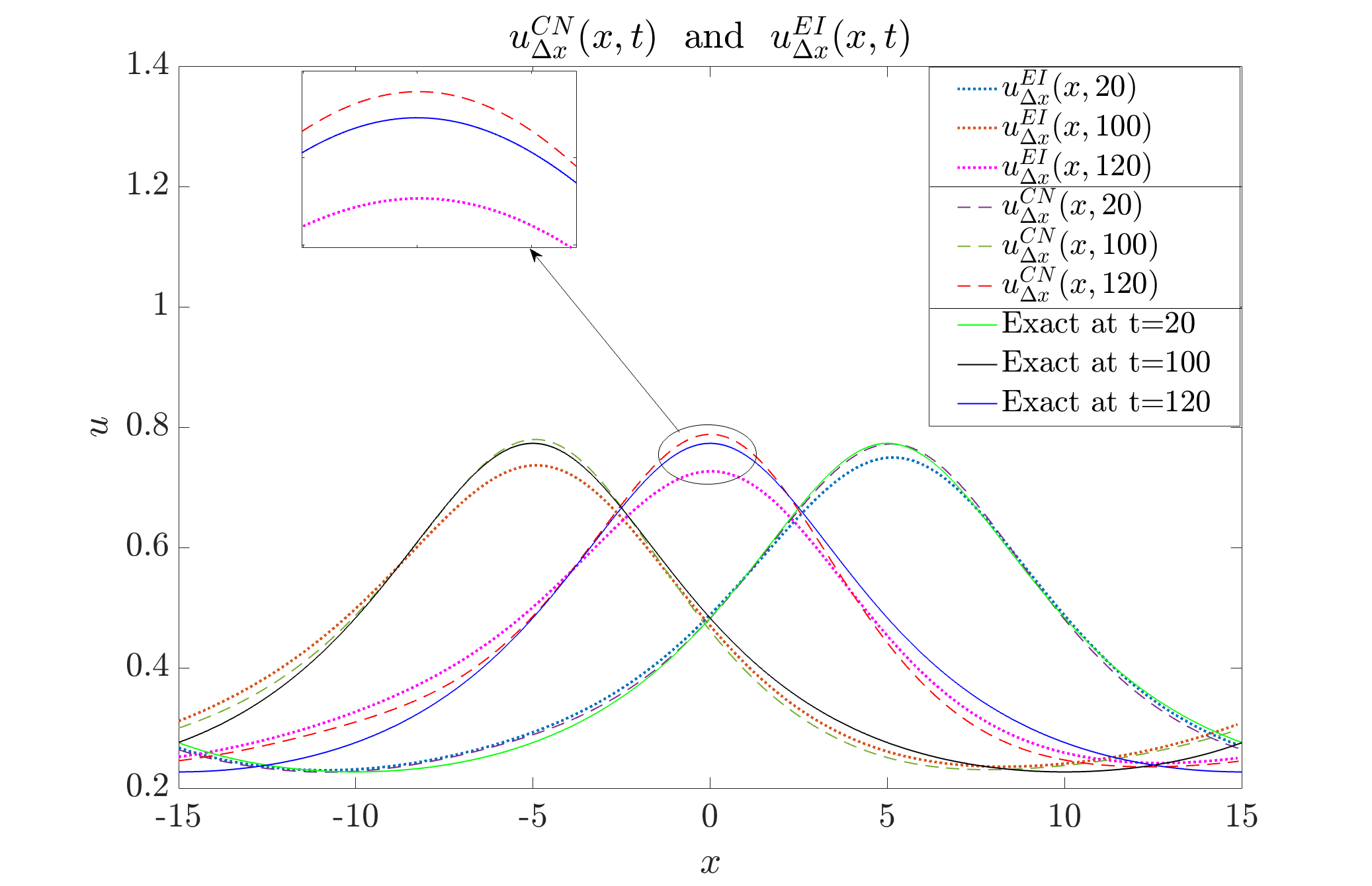}
    \caption{The exact solution $u_1$ and approximate solutions $u_{\Delta x}^{EI}$ and $u_{\Delta x}^{CN}$ at $t=20$, $t=100$ and $t=120$ with $N=512$ and $\alpha=1$.}
    \label{fig:BOoneCN}
\end{figure}

\begin{table}
  \centering
  \begin{tabular}{||c|c|c|c|c|c|c|c||}
        \hline
 N & $E_{EI}$ & $R_{EI}$ & $E_{CN}$ & $R_{CN}$ & $C_1^{\Delta}$ & $C_2^{\Delta}$ & $C_3^{\Delta}$ \\
 \hline
 \hline
   64  & 0.0216 & & 0.0650 &  & 1.052 & 2.47 & 1.052  \\
  &  & 1.100 & &  2.949 & & &\\
  128 & 0.0101 & &0.0084 && 1.003 & 1.04 & 1.04 \\
  &  &  0.947&   &  2.150 & &&\\
 256 & 0.0052 & &0.0019 & &1.000 & 1.000 & 1.000  \\
  &  & 1.056&  &   2.013 &&& \\
 512 & 0.0025 & &4.7034e-04 & &1.000 & 1.000 & 1.000 \\
   &  & 1.049&    & 2.004 &&& \\
 1024 & 0.0012 && 1.1720e-04 & &1.000 & 1.000 & 1.000 \\
  &  &  &  & && &\\
\hline
    \end{tabular}
   \caption{Convergence rates for approximate solutions with $\alpha= 1$.}
   \label{tab:TableBOCN}
\end{table}


\subsection{Fractional KdV equation}$\alpha =1.5$: \\
We assess the convergence using the initial condition $u_0(x) = 0.5 \sin(x)$ within the interval $[-4\pi,4\pi]$. Once again, by selecting $\alpha = 1.5$ in both schemes \eqref{FDscheme} and \eqref{CNscheme}, we compare the approximate solutions at time $t=5$ with the reference solution obtained using a higher number of grid points, $N=32000$, as the exact solution is unknown for this case. Table \ref{tab:Tablealpha15IMP} affirms the convergence of both schemes and quantities $C_i^{\Delta}$, $i=1,2,3$, are conserved for the Crank-Nicolson scheme \eqref{CNscheme}, and Figure \ref{fig:alpha15smoothCNFD} illustrates that the Crank-Nicolson scheme \eqref{CNscheme} converges more rapidly than the Euler implicit scheme \eqref{FDscheme}.
\begin{figure}
    \centering
    \includegraphics[width=0.95 \linewidth, height=8.5cm]{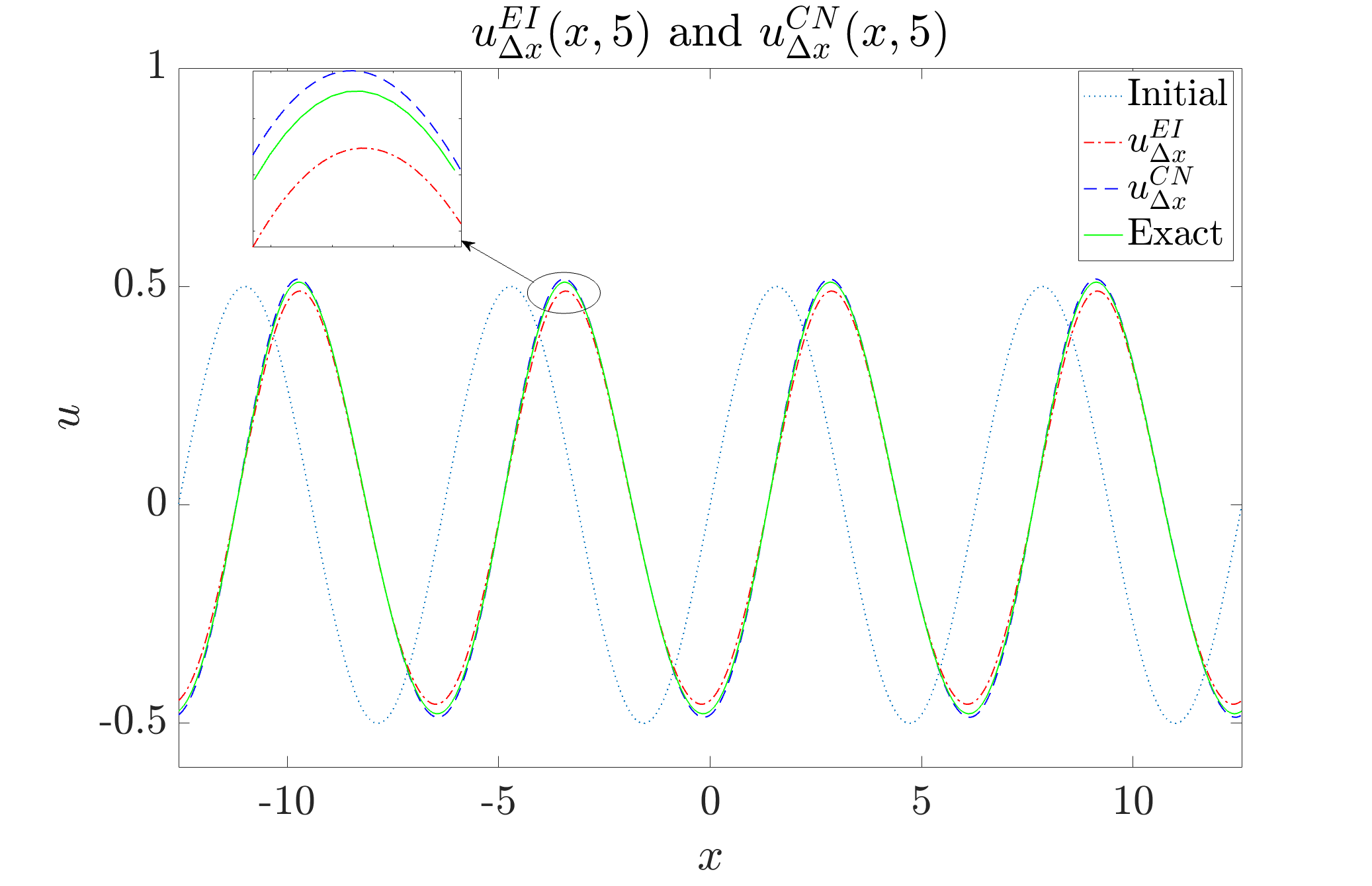}
    \caption{The reference and approximate solutions at $t=5$ with the initial data $u_0(x) = 0.5\sin(x)$, $N=1000$ grids and $\alpha=1.5$.}
    \label{fig:alpha15smoothCNFD}
\end{figure}
\begin{table}
  \centering
  \begin{tabular}{||c|c|c|c|c|c|c|c||}
        \hline
 N & $E_{EI}$ & $R_{EI}$ & $E_{CN}$ & $R_{CN}$ & $C_1$ & $C_2$ & $C_3$ \\
 \hline
 \hline
   250 & 0.5848 &  & 0.0733 & & 0.996 & 1.00  & 0.996 \\
  &  &  0.733&  &    1.418 &&&\\
  500  & 0.3517& & 0.0274 && 0.999& 1.00 &  0.999 \\
  & &  0.877 &  &    2.037&&&\\
 1000 & 0.1915  & & 0.0067& & 1.00 & 1.00 &  0.999 \\
  &  & 0.983&  &   1.979&&&\\
 2000 & 0.0969  && 0.0017 && 1.00 & 1.00 &1.00\\
   &  & 1.131 &   & 2.085&&& \\
 4000 & 0.0443  & & 0.0004 & &  1.00& 1.00 &1.00\\
  &  &  &  & &&& \\
\hline
    \end{tabular}
   \caption{Convergence rates for approximate solutions with $\alpha=1.5$.}
   \label{tab:Tablealpha15IMP}
\end{table}
   
\subsection{KdV equation}$\alpha \approx 2$: \\
Our objective is to analyze the behaviour of the solution of \eqref{fkdv} whenever the exponent $\alpha$ is close to $2$. As $\alpha=2$ correspond to the KdV equation, we compare the solutions generated by the schemes \eqref{FDscheme} and \eqref{CNscheme} by considering $\alpha\approx 2$ with the solution of KdV equation. For instance, the two-soliton solution for the KdV equation $u_t+\left(\frac{u^2}{2}\right)_x+u_{xxx}=0$ is explicitly introduced in \cite{holden2015convergence} as follows:
\begin{equation}\label{twosol}
     u_2(x,t) = 6(c-d) \frac{d \csch^2\left(\sqrt{d/2}(x-2dt)\right)+c \sech^2\left(\sqrt{c/2}(x-2ct)\right)}{\left(\sqrt{c}\tanh\left(\sqrt{c/2}(x-2ct)\right) - \sqrt{d} \coth\left(\sqrt{d/2}(x-2dt)\right)\right)^2}.
\end{equation}
Here $c$ and $d$ are any real parameters. We set the values of parameters $c=0.5$ and $d=1$. We use $u_0^{EI}(x,0) = u_2(x,-10)$ as the initial data for the scheme \eqref{FDscheme} and $u_0^{CN}(x,0) = u_2(x,-20)$ as the initial data for the scheme \eqref{CNscheme}. 


\begin{figure}
    \centering
    \includegraphics[width=0.9 \linewidth, height=10cm]{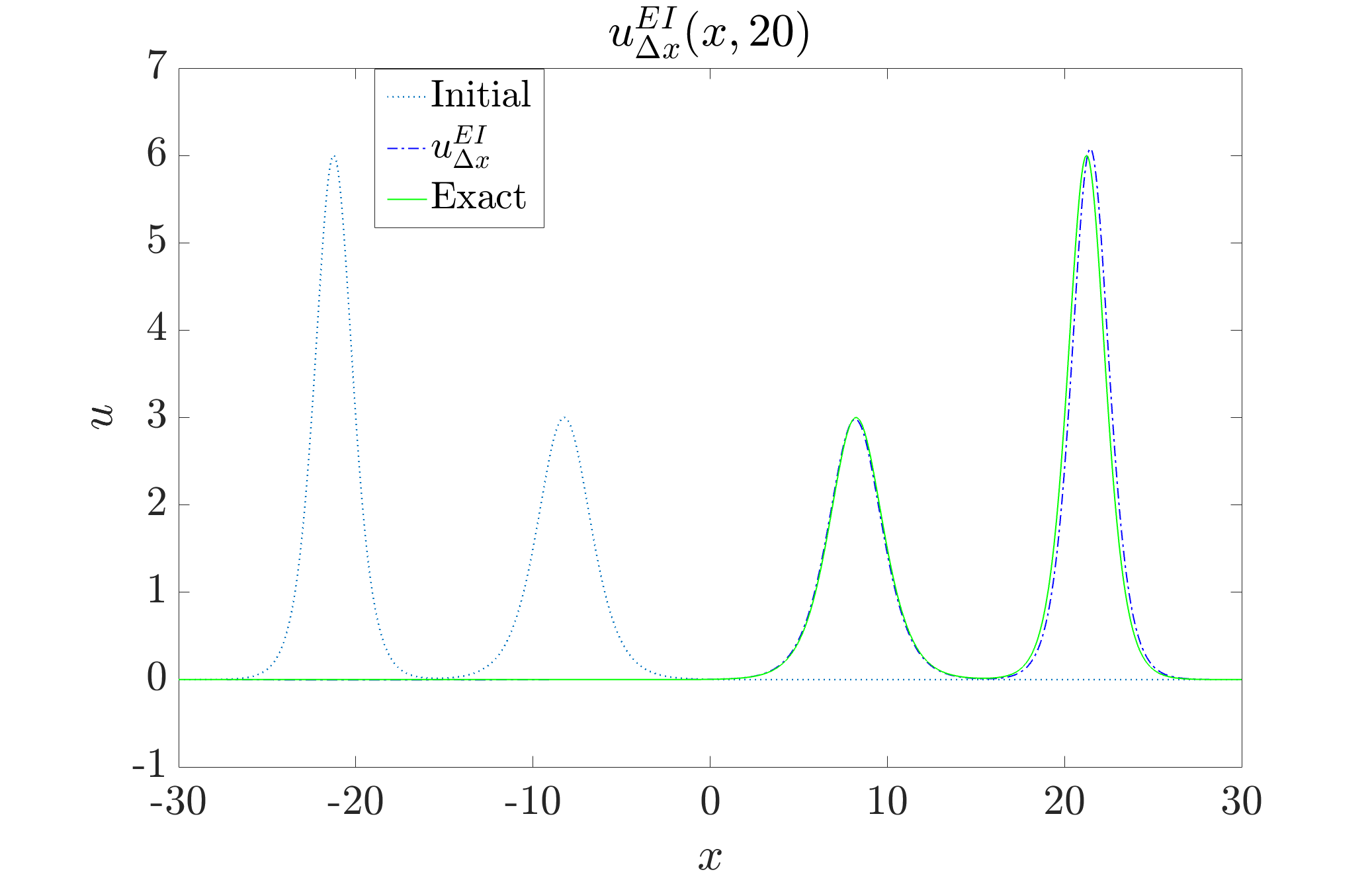}
    \caption{The exact(KdV) and numerical solution at $t=20$ with the initial data $w_2(x,-10)$ with $N=32000$ for Euler implicit.}
    \label{fig:KdVtwoim}
\end{figure}

\begin{table}
\begin{minipage}[b]{.45\textwidth}
  \centering
  \begin{tabular}{||c|c|c||}
        \hline
 N & E  & $R_E$ \\
 \hline
 \hline
   2000  & 2.5855  & \\
  &  &  1.313\\
  4000  & 1.0403  & \\
 &  &  1.029\\
 8000  & 0.5100  & \\
  &  & 1.086\\
 16000 & 0.2402 &\\
  &  & 1.048 \\
 32000 & 0.1161  &\\
  &  &  \\
\hline
    \end{tabular}
   \caption{Convergence rate of Euler implicit scheme when $\alpha\approx 2$.}
   \label{tab:TableaKDVIMP}
\end{minipage}
\begin{minipage}[b]{.50\textwidth}
  \centering
   \begin{tabular}{||c|c|c|c|c||}
        \hline
 N & E & $C_1$ & $C_2$ & $R_E$ \\
 \hline
 \hline
   250  & 1.1999 & 1.00 & 0.97 & \\
  &  &  &  &  2.942\\
  500  & 0.1561 & 1.00 & 1.01 & \\
  &  &  &  &  1.839\\
 1000  & 0.0436 & 1.00 & 1.00 & \\
  &  &  &  & 1.989\\
 2000 & 0.0110 & 1.00 & 1.00 &\\
   &  &  &  & 2.0809 \\
 4000 & 0.0026 & 1.00 & 1.00 &\\
  &  &  &  &  \\
 
\hline
    \end{tabular}
   \caption{Convergence rate of Crank-Nicolson scheme when $\alpha\approx 2$.}
   \label{tab:TableKDVCN}
\end{minipage}
\end{table}
\begin{figure}
    \centering
    \includegraphics[width=0.9 \linewidth, height=10cm]{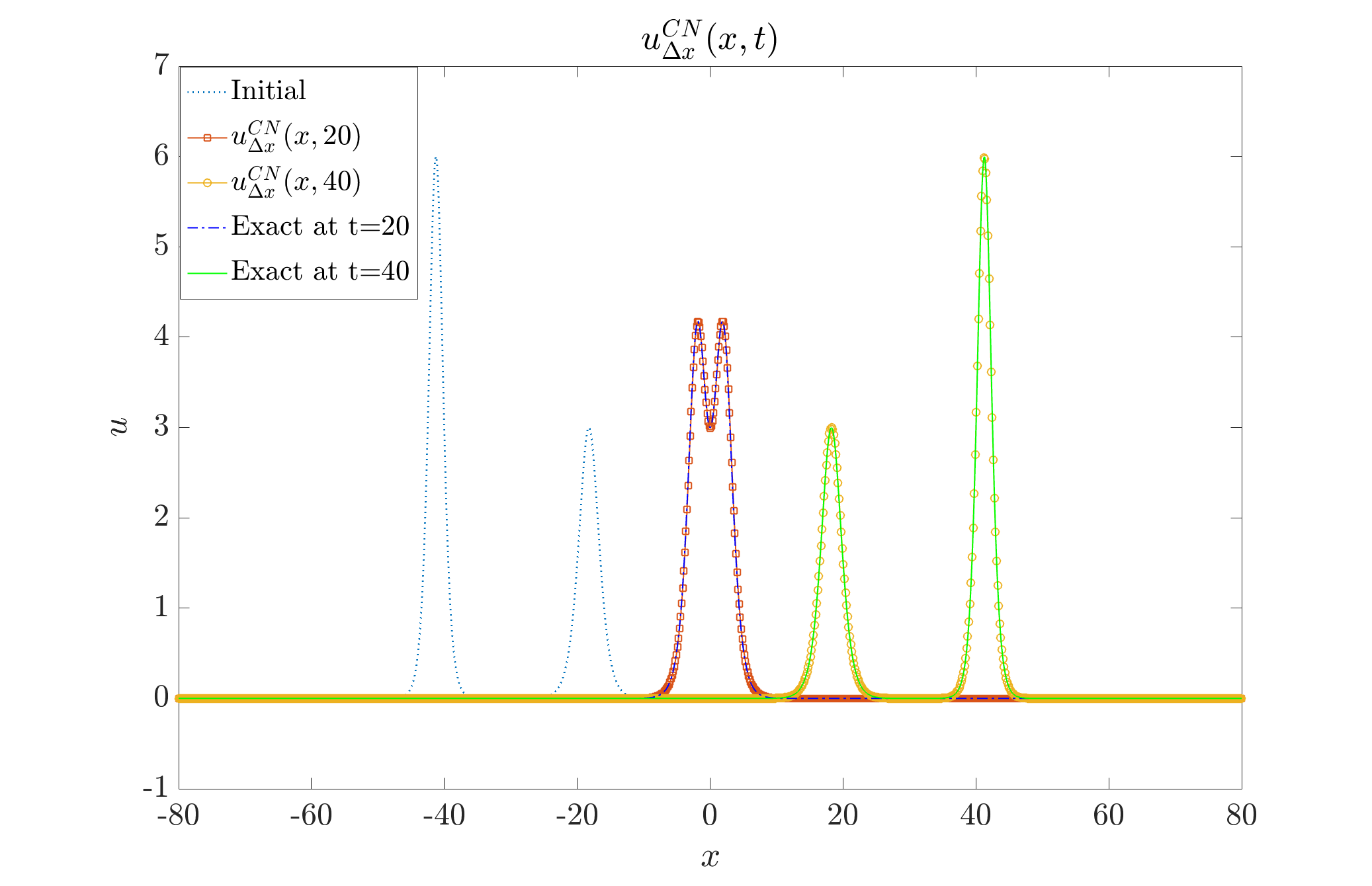}
    \caption{The exact and numerical solution obtained by Crank nicolson at $t=40$ with the initial data $w_2(x,-20)$ with $N=1000$.}
    \label{fig:KdVtwoCN}
\end{figure}

From a physical perspective, the taller soliton overtakes the shorter one, and both solitons emerge unaltered after the collision. This scenario presents a considerably more intricate computational challenge than solving for a single soliton solution. We choose $\alpha=1.999$ and compare the approximate solutions obtained by the schemes \eqref{FDscheme} and \eqref{CNscheme} to the exact solution of the KdV equation given by \eqref{twosol} at $t=20$ and $t=40$ respectively. Specifically, the approximate solution at $t=20$ obtained by the Euler implicit scheme \eqref{FDscheme} is denoted as $u_{\Delta x}^{EI}(x,20)$ and is compared with $u_2(x,10)$. The approximate solution at $t=40$ obtained by the Crank-Nicolson scheme \eqref{CNscheme} is denoted as $u_{\Delta x}^{EI}(x,40)$ and is compared with $u_2(x,20)$. Figure \ref{fig:KdVtwoCN} visually depicts that the taller soliton surpasses the smaller one at $t=20$, confirming the physical interpretation of soliton solutions. Table \ref{tab:TableaKDVIMP} and Table \ref{tab:TableKDVCN} display the convergence rates for both schemes. The Figure \ref{fig:KdVtwoim} and Figure \ref{fig:KdVtwoCN} provide graphical representations of the approximated and exact solutions of the KdV equation.

With the help of our several numerical illustrations we find that the scheme \eqref{CNscheme} converges faster than the scheme \eqref{FDscheme}. In addition, the scheme \eqref{CNscheme} preserves the conserved quantities with expected rates.
\section{Concluding remarks}\label{sec6}
In this study, we have rigorously investigated the stability and convergence properties of Euler implicit and a  Crank-Nicolson conservative finite difference scheme applied to the Cauchy problem associated with the fractional KdV equation. Notably, the Crank-Nicolson scheme exhibited superior convergence rates. Our comprehensive investigation into the stability, accuracy, and convergence of the proposed numerical schemes contributes valuable insights to the numerical analysis of the fractional KdV equation. The results presented herein not only enhance our understanding of the behavior of approximate solutions but also provide a solid foundation for future research in the numerical simulation of nonlinear dispersive equations involving the fractional Laplacian.

\end{document}